\documentclass{amsart}
\usepackage{amsmath, amsfonts} 
\usepackage{amsthm} 
\usepackage{amssymb}    
\usepackage{graphicx} 
\usepackage[top=1in, bottom=1in, left=1in, right=1in]{geometry}
\usepackage{mathtools}
\usepackage{mathrsfs}
\usepackage[colorlinks=true,linkcolor=black,anchorcolor=black,citecolor=black,filecolor=black,menucolor=black,runcolor=black,urlcolor=black]{hyperref}

\hypersetup{colorlinks=true,linkcolor=magenta,citecolor=blue}
\usepackage{enumerate}
\usepackage{longtable}
\usepackage{color,comment}
\usepackage{extarrows}
\usepackage{centernot}
\usepackage{tikz}
\usetikzlibrary{trees,arrows,positioning}
\usepackage{tikz-cd}

\usepackage{fancybox}
\usepackage{shadow}
\usepackage[all]{xy}
\usepackage{url}
\usepackage{mathrsfs}
\usepackage{float}
\usepackage{geometry}                		
\usepackage{graphicx}				
\usepackage{faktor}								
\usepackage{amssymb}
\usepackage{bbm}
\usepackage{hyperref}
\usepackage{tikz}
\usepackage{tikz-cd}
\theoremstyle{plain}
\newtheorem{theorem}{Theorem}[section]
\newtheorem*{theorem*}{Theorem}
\newtheorem{maintheorem}{Theorem}
\newtheorem{lemma}[theorem]{Lemma}
\newtheorem*{claim*}{Claim}
\newtheorem{proposition}[theorem]{Proposition}
\newtheorem{corollary}[theorem]{Corollary}

\theoremstyle{definition}

\newtheorem{remark}[theorem]{Remark}

\numberwithin{equation}{section}
\numberwithin{figure}{section}

\newcommand{\val}{\mathrm{val}}

	\newcommand{\chiE}{\chi^{\dagger}}
	
	\newcommand{\mat}[1]{\begin{pmatrix}#1\end{pmatrix}}

\begin{document}
	
	\title[Branching rules for principal series representations of  unramified $\mathrm{U}(1,1)$]{Branching rules for principal series representations of  unramified $\mathrm{U}(1,1)$}
	\author{Ekta Tiwari}
	\address{Department of Mathematics and Statistics, University of Ottawa, Ottawa, Canada K1N 6N5}
	\email{etiwa049@uottawa.ca}
	
\begin{abstract}
	Let $G$ denote the unramified quasi-split unitary group $\mathbb{U}(1,1)(F)$ over a $p$-adic field $F$ with residual characteristic $p \neq 2$. In this paper, we first construct a large family of irreducible representations of the maximal compact subgroup $\mathcal{K} = \mathbb{U}(1,1)(\mathcal{O}_F)$ of $G$. We then describe the branching rules for all principal series representations of $G$ upon restriction to $\mathcal{K}$ in terms of these representations. The resulting decomposition is multiplicity-free and is characterized by distinct degrees. Finally, we present two important applications of this decomposition that address certain recent open conjectures in the literature.  This is the first in series of two articles in which we provide branching rules for all irreducible smooth representations of the  
	$G$ upon restriction to $\mathcal{K}$. 
\end{abstract}
\maketitle
	\section{Introduction}
	The representation theory of connected reductive linear algebraic groups over non-archime\-d\-ean local fields is a vibrant subject at the intersection of algebra, number theory, harmonic analysis, and algebraic geometry. The study of representations of these groups offers fundamental insights into both local and global aspects of automorphic forms and the Langlands program, providing a bridge between local symmetry structures and global arithmetic properties. One powerful method for investigating the internal structure of such representations is through branching rules, which describe how irreducible representations decompose upon restriction to interesting subgroups.
	
	While branching rules have been extensively studied for real Lie groups, 
	the $p$-adic setting remains significantly less developed. 
	In the $p$-adic case $(p\neq 2)$, branching rules have been established for only a few groups, 
	including: 
	\(\mathrm{GL}_{2}(F)\), due to W.~Casselman~\cite{Cas71} and K.~Hansen~\cite{Kri87};  
	\(\mathrm{PGL}_{2}(F)\), due to A.~Silberger~\cite{Sil1970, Sil1977};  
	the Weil representation of \(\mathrm{Sp}_{2n}(F)\) restricted to one maximal compact subgroup, 
	due to D.~Prasad~\cite{Dip98}, with restriction to all conjugacy classes of maximal compact 
	subgroups subsequently obtained by K.~Maktouf and P.~Torasso~\cite{makpie2011};  
	principal series representations of \(\mathrm{GL}_{3}(F)\), were addressed by 
	M.~Nevins and P.~Campbell~\cite{CamMon2009, CamMon2010}, 
	with  explicit decomposition in terms of representations of maximal compact 
	subgroups later given by U.~Onn and P.~Singla~\cite{PooOnn2014};  
	and finally, the branching rules for principal series and supercuspidal representations 
	of \(\mathrm{SL}_{2}(F)\), obtained by M.~Nevins~\cite{Mon2005, MN13}.

 In order to study the branching, we divide the representations into three classes: the principal series representations, the depth-zero irreducible supercuspidal representations, and the positive-depth irreducible supercuspidal representations.  Together, these irreducible 
	representations generate the category of smooth representations of $G$. In this article, we determine the branching rules for  all
	principal series representations of $G$ upon restriction to a maximal compact 
	open subgroup $\mathcal{K}$ of $G$. The latter two classes are treated in~\cite{ET25}. Although $G$ is closely related to the $\mathrm{SL}(2,F)$, since its derived subgroup is isomorphic to $\mathrm{SL}(2,F)$, one can at best form a ballpark expectation for the decomposition based on the $\mathrm{SL}(2,F)$ case. Writing down the explicit decomposition for $G$ is considerably more delicate, and it is challenging to deduce it directly from the known results for $\mathrm{SL}(2,F)$. The paper of M.~Nevins on branching rules for $\mathrm{SL}(2,F)$~\cite{Mon2005} has served as a guiding framework for this work, and many of the resulting branching rules exhibit strong similarities.
	
We first use the properties of smooth representations to decompose the principal series representations upon restriction to $\mathcal{K}$, thereby obtaining a canonical decomposition. Let $\mathbbm{1}$ and $\mathrm{St}$ denote the inflations to $\mathcal{K}$ of the trivial and Steinberg representations, $\mathbbm{1}_{q}$ and $\mathrm{St}_{q}$, of $\mathrm{U}(1,1)(\mathfrak{f})$, respectively. Then the canonical decomposition can be stated as follows:

	\begin{maintheorem}[Theorem~\ref{thm2}]
		Let $\chi$ be a character of $T$ of minimal depth $r$, and let $(\pi_{\chi}, V_{\chi})$ denote the principal series representation associated to $\chi$. Then  for $d\geq r+1$, there exists irreducible representations $\mathcal{W}_{d, \chi}$ of degree $(q^{2}-1)q^{d-1}$ such that
		$$\mathrm{Res}_{\mathcal{K}}\pi_{\chi}\cong V_{\chi}^{\mathcal{K}_{r+1}}\oplus\bigoplus_{d \geq r+1}\mathcal{W}_{d,\chi}$$
		where $V_{\chi}^{\mathcal{K}_{1}}=\mathbbm{1}_{q}+\mathrm{St}_{q}$ if $\chi=\mathbbm{1}$, and  $V_{\chi}^{\mathcal{K}_{r+1}}$ is an irreducible representation of degree $(q+1)q^{r}$ otherwise.
	\end{maintheorem}
	\vspace{-1em}
The preceding theorem provides a decomposition; however, each representation $\mathcal{W}_{d, \chi}$ is realized as a quotient of an induced representation. To make this decomposition explicit, we construct irreducible representations of $\mathcal{K}$ that are isomorphic to each $\mathcal{W}_{d, \chi}$.
	Since maximal compact subgroups of reductive groups are not 
	themselves linear algebraic groups, comparatively little is known about 
	their representation theory. In fact, full classifications have only been 
	achieved in a few cases:  
	$\mathrm{PGL}_{2}(\mathcal{O})$ by A.~Silberger~\cite{Sil1970}, for $p\neq 2$;  
	$\mathrm{SL}_{2}(\mathcal{O})$ by J.~Shalika~\cite{Sha2004}, again under 
	the assumption $p\neq 2$;  
	$\mathrm{GL}_{2}(\mathcal{O})$ by A.~Stasinski~\cite{Sta2009}, valid for all 
	$p$;  
	and ramified $\mathrm{U}(1,1)(\mathcal{O}_F)$ by L.~G.~Frez~\cite{Frez2016}, 
	also for $p\neq 2$.  
	Earlier contributions include work of Kloosterman~\cite{Kl1946, Kl21946}, Tanaka~\cite{Tan1966, Tan1967} and Kutzko (unpublished) on 
	$\mathrm{SL}_{2}(\mathbb{Z}_p)$ for $p\neq 2$.  Nobs--Wolfart also provided a description of representations of $\mathrm{SL}_{2}(\mathbb{Z}_p)$ for all $p$~\cite{Nob1976, NobWolf1976}. Nagornyj provided a description  of representations of the group $\mathrm{GL}(2, \mathbb{Z}/p^{n}\mathbb{Z})$~\cite{Nag1981}. Beyond 
	these complete results, only partial progress is available in higher rank, 
	where families of representations of $\mathrm{GL}_{N}(\mathcal{O})$ have been 
	constructed by Aubert,~Onn,~Prasad,~Stasinski~\cite{Aubonnprasta2010}, Singla~\cite{Singla2010}, Crisp,~Meir,~Onn~\cite{Onn2024} and  Hill~\cite{ Hill1994, Hill21995}, though without full classification.

	To construct an irreducible representation of $\mathcal{K}$ of depth $d>0$, we select special elements $X$ in the Lie algebra $\mathfrak{k}$ of $\mathcal{K}$ whose centralizer $T(X)$ inside $\mathcal{K}$ is abelian.  
	Each such element $X$ determines a character $\Psi_{X}$ of depth $d$ of a subgroup $\mathcal{J}_{d}\subset\mathcal{K}$.
	
	\begin{maintheorem}[Theorem~\ref{Rep of K}]
		Let $\zeta$ be a character of $T(X)$ whose restriction to $T(X)\cap \mathcal{J}_{d}$ agrees with $\Psi_{X}$. 
		Let $\Psi_{X,\zeta}$ denote the unique extension of $\Psi_{X}$ and $\zeta$ to $T(X)\mathcal{J}_{d}$. 
		Then
		\[
		\mathcal{S}_{d}(X,\zeta)\;:=\;\mathrm{Ind}_{T(X)\mathcal{J}_{d}}^{\mathcal{K}}\Psi_{X,\zeta}
		\]
		is an irreducible representation of $\mathcal{K}$ of depth $d$ and degree $(q^{2}-1)q^{d-1}$.
	\end{maintheorem}
\vspace{-1em}	
Our principal result, stated precisely in Theorem~\ref{prop2} 
	and Corollary~\ref{thm3}, gives a complete description of the restriction of all principal series representations of $G$ to $\mathcal{K}$. 
	In terms of the notation introduced above, it may be summarized as follows.
	
	\begin{maintheorem}[Corollary~\ref{prop2}]
			Let $\chi$ be a character of \( T \) of minimal depth \( r \in \mathbb{Z}_{\geq 0} \). Then the restriction of the principal series representation \( \pi_{\chi} \) to $\mathcal{K}$ decomposes as direct sum of irreducible representations of $\mathcal{K}$ as follows
		\begin{equation*}
			\mathrm{Res}_{\mathcal{K}} \pi_{\chi} \cong
			\begin{cases}
				\mathrm{Ind}_{\mathcal{B}\mathcal{K}_{r+1}}^{\mathcal{K}} \chi \oplus \displaystyle\bigoplus_{d \geq r+1} \mathcal{S}_{d}(X, \zeta) & \text{if } r > 0, \\
				\mathrm{Ind}_{\mathcal{B}\mathcal{K}_{1}}^{\mathcal{K}} \chi \oplus \displaystyle\bigoplus_{d \geq 1} \mathcal{S}_{d}(X_{d}, \theta) & \text{if } r = 0 \text{ and } \chi \neq \mathbbm{1}, \\
				\mathbbm{1}_{q} \oplus \mathrm{St}_{q} \oplus \displaystyle\bigoplus_{d \geq 1} \mathcal{S}_{d}(X_{d}, \mathbbm{1}) & \text{if } r = 0 \text{ and } \chi = \mathbbm{1},
			\end{cases}
		\end{equation*}
			where:
		\begin{itemize}
			\item $	\mathrm{Ind}_{\mathcal{B}\mathcal{K}_{r+1}}^{\mathcal{K}} \chi $ is an irreducible representation of $\mathcal{K}$ of depth $r$;
			\item when $\chi$ has depth zero, $X_{d}$ is a nilpotent element and $\theta$ is the central character of $\pi_{\chi}$, while when $\chi$ has positive-depth, $X$ and $\zeta$ are twists of the datum used to construct $\pi_{\chi}$.
		\end{itemize}
	\end{maintheorem}
	

	More recently, there has been growing interest in understanding the decomposition of 
	representations of a group upon restriction to smaller neighborhoods of the identity, 
	with the expectation that such decompositions should involve representations of a simpler nature. 
	Some recent work in this direction includes results of M.~Nevins on complex representations 
	of $\mathrm{SL}(2,F)$ for $p \neq 2$~\cite{Mon2024}, and the analysis of depth-zero 
	supercuspidal representations of $\mathrm{SL}(2,F)$ in characteristic $2$, due to 
	Z.~Karaganis and M.~Nevins~\cite{ZanMon2025}. 
	Related questions were also studied by G.~Henniart and M.~F.~Vignéras 
	\cite{henniart2024representationssl2f, guy2024representationsglndnearidentity}, 
	who considered representations of $\mathrm{SL}(2,F)$ and of inner forms of $\mathrm{GL}(n,F)$ 
	over fields $\mathcal{R}$ of characteristic different from $p$. These ideas lead to a set of conjectures two of which we solve for $\mathbb{U}(1,1)(F)$. In this article, we present the solution for principal series representations, the second article provides analogous results for the other two classes of representations of $G$.

	To state our solution, we first observe that the higher-depth components appearing in the restriction of any irreducible principal series representation $\pi$ of $G$ already coincide with those arising from depth-zero principal series representations with the same central character after potentially twisiting by a character of $G$.
	
	\begin{maintheorem}[Theorem~\ref{keyidentificationp}]
	Let $\chi$ be a character of \( T \) of minimal depth \( r \in \mathbb{Z}_{\geq 0} \) and let $\pi_{\chi}$ be the associated principal series representation  with depth-zero central character $\theta$. Then, for all $d>2r$,
		\[
		\mathcal{S}_{d}(X, \zeta)\cong \mathcal{S}_{d}(X_{d}, \theta),
		\]
		where $X_{d}$ is a nilpotent element.
	\end{maintheorem}	
\vspace{-1em}	
	In fact, we show that one can choose two depth-zero principal series representations such that the higher-depth components of any principal series representation $\pi_{\chi}$, upon restriction to $\mathcal{K}$, coincide with the higher-depth components, upon restriction to $\mathcal{K}$, of one these two fixed representations, up to twisting by a character of $G$.

	For the other conjecture, we note our unitary group $G$ admits exactly two non-zero nilpotent $G$-orbits in its Lie algebra $\mathfrak{g}$, which we denote by $\mathcal{N}_{1}$ and $\mathcal{N}_{\varpi}$. To each of these orbits, we associate a highly reducible representation
	\[
	\begin{aligned}
		\tau_{\mathcal{N}_{1}}(\theta) &= \bigoplus_{d \in 2\mathbb{Z}_{> 0}} \mathcal{S}_{d}(X_{d}, \theta), \qquad
		\tau_{\mathcal{N}_{\varpi}}(\theta) = \bigoplus_{d \in 2\mathbb{Z}_{\geq 0} + 1 } \mathcal{S}_{d}(X_{d}, \theta),
	\end{aligned}
	\]
	where $\theta$ is a character of the center of $G$. Then, in the Grothendieck group of representations, we obtain the following decomposition.
	
	\begin{maintheorem}[Theorem~\ref{principalseriesrestrictedtoK2r}]
	Let $\chi$ be a character of \( T \) of minimal depth \( r \in \mathbb{Z}_{\geq 0} \), let $\pi_{\chi}$ be the associated principal series representation, and let $\theta$ denote its central character. Then 
		\[
		\mathrm{Res}_{\mathcal{K}_{2r+1}}\pi
		\;=\;
		(q+1)\, \mathbbm{1}
		\;\oplus\;
		a_{\mathcal{N}_{1}}\,\mathrm{Res}_{\mathcal{K}_{2r+1}}\bigl( \tau_{\mathcal{N}_{1}}(\theta) \bigr)
		\;\oplus\;
		a_{\mathcal{N}_{\varpi}}\,\mathrm{Res}_{\mathcal{K}_{2r+1}}\bigl( \tau_{\mathcal{N}_{\varpi}}(\theta) \bigr),
		\]
		where $a_{\mathcal{N}_{1}}, a_{\mathcal{N}_{\varpi}} \in \{0,1\}$.
	\end{maintheorem}
\vspace{-1em}	
	Thus every constituent of $\mathrm{Res}_{\mathcal{K}_{2r+}}\pi$ is either the trivial representation or constructed from nilpotent elements $X_{d}$ in the Lie algebra of $G$.

	This article is organized as follows. In~\S\ref{notation and background}, we present the necessary background, including results on \(p\)-adic fields, the structure theory of the group \(G\), and several classical results and definitions from the representation theory of \(p\)-adic groups, together with examples adapted to our setting. In~\S\ref{construction of representations and reduction lemmas}, we describe all principal series representations of \(G\) and establish several reduction lemmas. In~\S\ref{a canonical decomposition}, we obtain a canonical decomposition into irreducible \(\mathcal{K}\)-representations. In~\S\ref{chapter3}, we construct explicit irreducible representations of the maximal compact subgroup \(\mathcal{K}\), adapting Shalika’s method~\cite{Sha2004} to the setting of \(G\). In~\S\ref{an explicit decomposition}, we present an explicit decomposition of all principal series representations upon restriction to \(\mathcal{K}\) in terms of these representations. We conclude in~\S\ref{restrictiontoK2rprincipalseries} with applications of our main results.
	\vspace{-1em}
	
	\subsection*{Acknowledgements}
	I would like to thank  Monica Nevins, my PhD supervisor, for suggesting this problem and for her invaluable guidance and moral support throughout the development of this work. This paper could not have been written without her kind assistance. I would also like to thank the University of Ottawa for providing a stimulating research environment.

\section{Notations and background}\label{notation and background}
\subsection{The field}\label{the field}	Let $F$ be a non-archimedean local field, and let $E = F[\sqrt{\epsilon}]$ be the quadratic unramified extension of $F$, where $\epsilon$ is a non-square element of $\mathcal{O}_{F}^{\times}$.  We fix $\varpi$ a uniformiser of $F$ and normalize the valuation $\nu$ on $F$ such that $\nu(\varpi) = 1$. The valuation $\nu$ is extended to $E$, and we use the same symbol $\nu$ for this extension. We define the \emph{absolute value} (also called the analytic norm) associated with $\nu$, denoted by $|\cdot|$, as
\[
|x| := q^{-\nu(x)}, \quad x \in E,
\]
where $q = |\mathcal{O}_{F}/\mathfrak{p}_{F}|$ is the cardinality of the residue field of $F$. We denote the ring of integers of $F$ by $\mathcal{O}_{F}$ and that of $E$ by $\mathcal{O}_{E}$.  
Their respective maximal ideals are denoted by $\mathfrak{p}_{F}$ and $\mathfrak{p}_{E}$.  
The residue fields of $F$ and $E$ are given by $\mathfrak{f} \cong \mathcal{O}_{F}/\mathfrak{p}_{F}$ and $\mathfrak{e} \cong \mathcal{O}_{E}/\mathfrak{p}_{E},$ respectively.  
Since  the cardinality of $\mathfrak{f}$ is $q$, the cardinality of $\mathfrak{e}$ is $q^{2}$. Let $\sigma$ denote the nontrivial element of $\mathrm{Gal}(E/F)$. For $x \in E$, we write $\overline{x} := \sigma(x)$. The \emph{norm} and \emph{trace} from $E$ to $F$ are defined by
\[
\mathrm{N}_{E/F}(x) = x \, \overline{x}, \quad  \mathrm{Tr}_{E/F}(x) = x + \overline{x}, \quad \text{for\;all } x \in E.
\]
We denote by \( E^1 \subset E^\times \) the subgroup consisting of elements of norm one, \emph{i.e.}, $E^1 := \{ x \in E^\times \mid \mathrm{N}_{E/F}(x) = 1 \}.$
The norm map induces a surjection $\mathrm{N}_{E/F} \colon \mathcal{O}_E^\times \twoheadrightarrow \mathcal{O}_F^\times,$ and  this filters to a surjection $\mathrm{N}_{E/F} \colon 1 + \mathfrak{p}_E^d \twoheadrightarrow 1 + \mathfrak{p}_F^d$, for every $d\in\mathbb{Z}_{\geq 1}$. We fix an additive character $\psi'$ of $F$, which is trivial on $\mathfrak{p}_{F}$ but nontrivial on $\mathcal{O}_{F}$. We then define a character $\psi$ of $E$ by  
	\[
	\psi(x) = \psi^{\prime}\left(\frac{x + \overline{x}}{2}\right), \quad \text{for all } x \in E.
	\]
	Note that $\psi$ is also trivial on $\mathfrak{p}_{E}$ and nontrivial on $\mathcal{O}_{E}$.

\subsection{Some structure theory}\label{the group} We adopt the convention of using blackboard bold font to denote an algebraic group \( \mathbb{G} \) defined over \( F \), and the corresponding roman letter for its group of \( F \)-rational points, that is, \( G = \mathbb{G}(F) \). For simplicity of notation, we may refer to a group \( T \) as a ``maximal torus of \( G \)" when we mean that \( T \) is the group of \( F \)-rational points of an algebraic torus \( \mathbb{T} \subset \mathbb{G} \) that is defined over $F$.

Given subsets \( L_i \) of \( F \) or \( E \), we use the shorthand notation
\[
\left\{\begin{pmatrix}L_1 & L_2 \\ L_3 & L_4 \end{pmatrix}\right\}
:= 
\left\{
\begin{pmatrix}
	a_1 & a_2 \\
	a_3 & a_4
\end{pmatrix}
\;\middle|\; a_i \in L_i
\right\}.
\]
For \( g \in G \) and a subgroup \( K \subset G \), we write \( K^g := gKg^{-1} \) and  \( {}^gK := g^{-1}Kg \).

\subsubsection{The group $\mathrm{U}(1,1)$} Let $\mathbb{G}$ denote the quasi-split unitary group $\mathbb{U}(1,1)$ whose group of $F$-points is  
\begin{align*}
	G:=\mathbb{G}(F) = \left\{g \in M_{2}(E) \mid \overline{g}^{\top}\mathrm{w}g = \mathrm{w}\right\},
\end{align*}
where $\mathrm{w}=\begin{pmatrix}0& 1\\ 1&0\end{pmatrix}$, and  $\overline{g}^{\top}=(\overline{g_{ij}})^{\top}=(\overline{g_{ji}})$.
More explicitly, $G$ has the form
\begin{align}\label{definitionofG}
	G= & \left\{
	\begin{pmatrix}
		a & b \\
		c & d
	\end{pmatrix}
	\in M_{2}(E)
	\; \middle| \;
	\begin{aligned}
		& \overline{a}d + \overline{c}b = 1, \\
		& \overline{a}c, \overline{b}d \in \sqrt{\epsilon}F
	\end{aligned}
	\right\}.
\end{align}	
The Lie algebra $\mathfrak{g}$ of $G$ is defined as $\{X\in \mathfrak{gl}(2,E)\;|\;\overline{X}^{\top}\mathrm{w}+\mathrm{w}X=0\}$ and thus has the form
$$
\mathfrak{g} := \mathfrak{u}(1,1)(F) = \left\{
\begin{pmatrix}
	u+\sqrt{\epsilon}v & y\sqrt{\epsilon} \\
	z\sqrt{\epsilon} & -u+\sqrt{\epsilon}v
\end{pmatrix}
\;\middle|\; u, v, y, z \in F
\right\}.
$$ 
We use $\mathfrak{g}^{*}$ to denote the algebraic dual of the Lie algebra of $\mathfrak{g}$.
The center of $G$ is the subgroup $Z$ of scalar matrices, which we identify with $E^{1}$. 
	
Let $\mathbb{T}$ be the diagonal torus of $\mathbb{U}(1,1)$, whose group of $F$-points is 			
\begin{equation}\label{definitionofmaximallyFsplitFtorus}
	T:=\mathbb{T}(F)=\left\{\begin{pmatrix}
		a&0\\ 0&\overline{a}^{-1}
	\end{pmatrix}\mid a\in E^{\times}\right\}.\end{equation} 
We have $\mathbb{T}(F)\cong E^\times$ is a maximally $F$-split maximal $F$-torus of $G$.  Any other maximal torus of $G$ is a $G$-conjugate of $T$. The maximal $F$-split subtorus of $\mathbb{T}$ is $\mathbb{S}$ where
\begin{equation}\label{definitionofS} S:=\mathbb{S}(F)=\left\{\begin{pmatrix}
	a&0\\ 0&a^{-1}
\end{pmatrix}\mid a\in F^{\times}\right\}.\end{equation}
Similar to the case of \(\mathrm{SL}(2,F)\), the group \(G\) also has two conjugacy classes of maximal compact subgroups, represented by 
$$\mathcal{K}
	=\mathbb{G}(\mathcal{O}_{E})
	=\begin{pmatrix}
		\mathcal{O}_{E} & \mathcal{O}_{E} \\
		\mathcal{O}_{E} & \mathcal{O}_{E}
	\end{pmatrix}\cap G\;\hspace{2em}\text{and}\;\hspace{2em}\mathcal{K}^{\eta}
	=\begin{pmatrix}
		\mathcal{O}_{E} & \mathfrak{p}_{E}^{-1} \\
		\mathfrak{p}_{E} & \mathcal{O}_{E}
	\end{pmatrix}\cap G,$$
where 
	\(\eta=\begin{psmallmatrix}1&0\\ 0&\varpi\end{psmallmatrix}\in\mathrm{GL}(2,E)\) is such that it normalizes our unitary group \(G\).
	\begin{lemma}\label{Kproperty}
		Let $\begin{pmatrix}x&y \\z&w\end{pmatrix}\in \mathcal{K}$. Then $x, w\in\mathcal{O}_{E}^{\times}$ or $y, z\in \mathcal{O}_{E}^{\times}$.
	\end{lemma}	
	\begin{proof}
		Note that $\mathrm{det}(k)\in\mathcal{O}_{E}^{\times}$ for all $k\in\mathcal{K}$, as otherwise the inverse wouldn't be in the group. Therefore, if $k:=\begin{pmatrix}x&y \\z&w\end{pmatrix}\in \mathcal{K}$, then we have
		$0=\nu(xw-yz)\geq \mathrm{min}\{\nu(xw), \nu(yz)\}\geq 0.$
		Thus at least one of $xw$ or $yz$ must have valuation $0$, and hence $x,w\in\mathcal{O}_{E}^{\times}$ or $y,z\in\mathcal{O}_{E}^{\times}$.
	\end{proof}
	
We denote by $\mathcal{B}$ the intersection of the Borel subgroup $B$ of $G$ with $\mathcal{K}$.

	\begin{proposition}\label{double coset for principal series}
		A set of double coset representatives for $\mathcal{B} \backslash \mathcal{K} / \mathcal{B}$ is
		\[
		\left\{
		\begin{pmatrix}1&0\\ 0&1\end{pmatrix},\quad
		\begin{pmatrix}0&1\\ 1&0\end{pmatrix},\quad
		\begin{pmatrix}1&0\\ \sqrt{\epsilon}\,\varpi^{k}&1\end{pmatrix}
		\;\middle|\; k \ge 1
		\right\}.
		\]
	\end{proposition}
	\begin{proof}
		Let $k:=\begin{pmatrix}
			x & y \\
			z & w
		\end{pmatrix} \in \mathcal{K}$. If $z=0$, then $k$ belongs to the identity coset $\mathcal{B}I\mathcal{B}$. If $z \neq 0$ but $z\in \mathfrak{p}_{E}$, then Lemma \ref{Kproperty} yields $w\in\mathcal{O}_{E}^{\times}$. Since $w^{-1} = \frac{\overline{w}}{w\overline{w}}$ and $\overline{w}y\in\sqrt{\epsilon}F$, we conclude that $w^{-1}y\in \sqrt{\epsilon}F$. We also have $x\overline{w}+z\overline{y}=1$, implying that $x=\overline{w}^{-1}-\overline{w}^{-1}\overline{y}z=\overline{w}^{-1}+w^{-1}yz$. Hence we can express $k$ as
			\[
		\begin{pmatrix}
			x & y \\
			z & w
		\end{pmatrix} = \begin{pmatrix}
			\overline{w}^{-1} & y \\
			0 & w
		\end{pmatrix} \begin{pmatrix}\gamma^{-1}&0\\0&\overline{\gamma}\end{pmatrix}\begin{pmatrix}1&0\\ \sqrt{\epsilon}\varpi^{k}&1\end{pmatrix}\begin{pmatrix}\gamma&0\\0&\overline{\gamma}^{-1}\end{pmatrix}\in\mathcal{B}\begin{pmatrix}1&0\\ \sqrt{\epsilon}\varpi^{k}&1\end{pmatrix}\mathcal{B},
		\]
		where $k=\nu(w^{-1}z)$ and $\gamma\overline{\gamma}=\frac{w^{-1}z}{\varpi^{k}\sqrt{\epsilon}}$. Lastly, if  $z\in \mathcal{O}_{E}^{\times}$, then we have $z^{-1}x\in \sqrt{\epsilon}F$ and $y=\overline{z}^{-1}+z^{-1}xw$. Therefore,  we can express $k$ as
		\[
		\begin{pmatrix}
			x & y \\
			z & w
		\end{pmatrix} = \begin{pmatrix}
			\overline{z}^{-1} & x \\
			0 & z
		\end{pmatrix}\begin{pmatrix}0&1 \\ 1&0\end{pmatrix} \begin{pmatrix}
			1& z^{-1}w \\
			0 & 1
		\end{pmatrix}\in \mathcal{B}\mathrm{w}\mathcal{B},
		\]
		Finally, we show that the double cosets represented by
		$\begin{psmallmatrix}1&0\\ \sqrt{\epsilon}\,\varpi^{k}&1\end{psmallmatrix}$ are pairwise distinct as $k$ varies positive integers greater that or equal to $1$. Suppose for contradiction that $k>l\ge 1$ and
		\[
		\begin{pmatrix}1&0\\ \sqrt{\epsilon}\,\varpi^{k}&1\end{pmatrix}
		=
		\begin{pmatrix}x & y \\ 0 & \overline{x}^{-1}\end{pmatrix}
		\begin{pmatrix}1&0\\ \sqrt{\epsilon}\,\varpi^{l}&1\end{pmatrix}
		\begin{pmatrix}z & t \\ 0 & \overline{z}^{-1}\end{pmatrix},
		\qquad
		x,z\in \mathcal{O}_E^{\times},\ y,t\in \mathcal{O}_E.
		\]
		Comparing $(2,1)$-entries yields $\overline{x}^{-1}z=\varpi^{k-l}\in \mathfrak{p}_F$, which is impossible since $\overline{x}^{-1}z\in \mathcal{O}_E^{\times}$. Hence the double cosets are distinct.
\end{proof}

\subsubsection{Filtrations on $T$ and on $\mathcal{K}$}
For $r=0$, we define
\begin{align*}
	T_{0}&=\left\{
	\begin{pmatrix}
		a & 0 \\
		0 & \overline{a}^{-1}
	\end{pmatrix}
	\;\middle|\; a\in \mathcal{O}_{E}^{\times}\right\}.
\end{align*}
$T_{0}$ has a nice filtration indexed by non-negative integers. In general, these filtration subgroups are rather difficult to compute. However, in our case they admit a nice description, given by 			
\begin{align*}
T_{r}&=\left\{
	\begin{pmatrix}
		a & 0 \\
		0 & \overline{a}^{-1}
	\end{pmatrix}
	\;\middle|\; a\in 1+\mathfrak{p}_{E}^{\lceil r \rceil}\right\}, 
	\qquad \text{for } r>0.
\end{align*}

The maximal compact subgroup $\mathcal{K}$ also admits a filtration by normal subgroups indexed by nonnegative integers 
\[
\mathcal{K} = \mathcal{K}_{0} \supsetneq \mathcal{K}_{1} \supsetneq \mathcal{K}_{2} \supsetneq \mathcal{K}_{3} \supsetneq \cdots;
\]  
 for each $n \in \mathbb{Z}_{\geq 0}$, these subgroups are given by  
\[
\mathcal{K}_{n} = 
\begin{pmatrix}
	1 + \mathfrak{p}_{E}^{n} & \mathfrak{p}_{E}^{n} \\ 
	\mathfrak{p}_{E}^{n} & 1 + \mathfrak{p}_{E}^{n}
\end{pmatrix} \cap G.
\]  
These are the Moy--Prasad filtration subgroups of  the parahoric subgroup $\mathcal{K}$ of $G$. The collection $\{\mathcal{K}_{n}\}_{n \geq 0}$ forms a neighborhood basis of the identity in $\mathcal{K}$. The quotient $\mathcal{K}/\mathcal{K}_{0+}$ is isomorphic to the $\mathfrak{f}$ points of our unitary group $\mathbb{U}(1,1)$.	For $r\in \mathbb{R}$, we can similarly define a filtration $\mathfrak{k}_{r}$ of the Lie algebra  $\mathfrak{k}$ of $\mathcal{K}$ 
\begin{equation}
	\mathfrak{k}_{ r}=\begin{psmallmatrix}
		\mathfrak{p}_{E}^{\lceil r \rceil}&\sqrt{\epsilon}\mathfrak{p}_{F}^{\lceil r\rceil}\\
		\sqrt{\epsilon}\mathfrak{p}_{F}^{\lceil r\rceil}& \mathfrak{p}_{E}^{\lceil r \rceil}
	\end{psmallmatrix}\cap \mathfrak{k}. 
\end{equation}		
We define the filtration subspace $\mathfrak{k}^{*}_{-r}$ of the dual of the Lie algebra by
\begin{equation}
	\mathfrak{k}^{*}_{-r}=\{\lambda\in\mathfrak{k}^{*}\mid \lambda(Y)\in\mathfrak{p}_{E}, \;\forall Y\in\mathfrak{k}_{s},\; s>r\}.\label{filtrationofdual}
\end{equation}		
We can identify $\mathfrak{k}^{*}$ with $\mathfrak{k}$ by using the trace form.
In particular, for $X\in\mathfrak{k}$, we define $\lambda_{X}\in\mathfrak{k}^*$ by the equation
$\lambda_{X}(Y)=\mathrm{Tr}(XY)$ for all $Y\in\mathfrak{k}$. Under this identification, $\mathfrak{k}^{*}_{r}=\mathfrak{k}_{r}$ for all $r\in\mathbb{R}_{\geq 0}$.
\subsection{Representation theory}\label{representation theory}In this section, we recall some basic definitions and concepts from representation theory specializing to the case of our unitary group $G$ (see~\cite[\S I]{Car79}).

Let \((\pi, V)\) be a complex representation of \(G\), and for a compact open subgroup \(K \subset G\), set 
\[
V^{K} = \{ v \in V \mid \pi(k)v = v \text{ for all } k \in K \}.
\]
Then \((\pi, V)\) is \emph{smooth} if and only if \(V = \bigcup_{n \ge 0}^{\infty} V^{\mathcal{K}_{n}}\), where $\mathcal{K}=\mathbb{U}(1,1)(\mathcal{O}_{F})$. It can be shown that if \((\pi, V)\) is a smooth representation of $G$, then each subspace \(V^{\mathcal{K}_{n}}\) is \(\mathcal{K}\)-invariant, and any irreducible \(\mathcal{K}\)-subrepresentation of \(V\) is a subrepresentation of \(V^{\mathcal{K}_{n}}\) for some \(n \geq 0\). Hence, the decomposition of \((\pi, V)\) upon restriction to \(\mathcal{K}\) reduces to decomposing each \(V^{\mathcal{K}_{n}}\).


We now recall the definition of induced representations. Given  a closed subgroup \( H \) of \( G \), and a representation \( (\rho, W) \) of \( H \), one can define a representation \( \left(\mathrm{Ind}_{H}^{G} \rho, \mathrm{Ind}_{H}^{G} W\right) \) of \( G \), called the \emph{induced representation}, where
\[
\mathrm{Ind}_{H}^{G} W = \left\{ f : G \rightarrow W \;\middle|\; 
\begin{array}{l}
	f \text{ is locally constant, and} \\
	f(hg) = \rho(h)f(g) \text{ for all } h \in H,\, g \in G
\end{array}
\right\},
\]
and the action of \( G \) is given by $(g\cdot f)(x)=f(xg)$ for all $x, g \in G$.

\begin{lemma}\label{principalseriescharacterreduction5}
	Let $(\lambda, \mathbb{C})$ be a character of $G$, and let $(\rho, V)$ be an irreducible representation of a subgroup $H$ of $G$. Then \[\mathrm{Ind}_{H}^{G}(\lambda \otimes \rho)\;\cong\;\lambda \otimes \mathrm{Ind}_{H}^{G}\rho.\]
\end{lemma}
%
%
%

We now review two important results that we will use quite frequently~\cite[\S2.4]{Bus06}.
	\begin{theorem}[Frobenius Reciprocity] Let $H$ be a closed subgroup of $G$,  $(\rho, W)$ a smooth representation of $H$ and  $(\pi, V)$ a smooth representation of $G$. Then 
	\begin{align*}
		\mathrm{Hom}_{G}(\pi, \mathrm{Ind}_{H}^{G}\rho)&\cong \mathrm{Hom}_{H}(\pi, \rho)
	\end{align*}	
\end{theorem}	

Similarly, for open subgroups $H$ of $G$, compact induction has its own form of Frobenius Reciprocity property.  
\begin{theorem} Let $H$ be an open subgroup of $G$, let $(\rho, W)$ be a smooth representation of $H$ and $(\pi, V)$ a smooth representation of $G$. Then
	\begin{align*}
		\mathrm{Hom}_{G}( c  \scalebox{0.9}[0.9]{-} \mathrm{Ind}_{H}^{G}\rho, \pi)&\cong \mathrm{Hom}_{H}(\rho, \pi)
	\end{align*}	
	is an isomorphism which is functorial in both variables $\pi,\;\rho$.
\end{theorem}

	\begin{remark} Note that if $G/ H$ is compact then $c  \scalebox{0.9}[0.9]{-} \mathrm{Ind}_{H}^{G}\rho=\mathrm{Ind}_{H}^{G}\rho$. Thus, in that case, induction is both left and right adjoint to restriction.
	\end{remark}

In our setting, the groups under consideration are infinite. However, for the maximal compact subgroup $\mathcal{K}$ we have $\mathcal{K}/\mathcal{K}_{0+} \;\cong\; \mathbb{U}(1,1)(\mathfrak{f}),$
and, more generally, $\mathcal{K}/\mathcal{K}_{r+}$ is finite for every $r > 0$. Every irreducible smooth representation $(\pi, V)$ of $\mathcal{K}$ factors through  $\mathcal{K}/\mathcal{K}_{r+}$ for some $r\geq0$; the least of such $r$ is called the \emph{depth} of $\pi$. We will identify a representation of  $\mathcal{K}/\mathcal{K}_{r+}$ with its inflation to $\mathcal{K}$ without comment. It is worth noting that inflation and induction commute with each other.

\section{Construction of representations}\label{construction of representations and reduction lemmas}
Let $B$ be the Borel subgroup of upper triangular matrices in $G$. Write $B=TU$
where
$T$ is a maximally $F$-split maximal $F$-torus $T$ of $G$ as in~\eqref{definitionofmaximallyFsplitFtorus}, and 
$$U=\left\{\begin{pmatrix}1&\sqrt{\epsilon}b\\ 0&1\end{pmatrix}\;\middle\vert\;b\in F\right\} $$ is the unipotent radical of $B$. Since the commutator of $B$ is $U$, a character of $B$ is defined by its restriction to $T$, and any character of $T$ can be extended trivially to a character of $B$. When convenient, identify $\chi$ with a character $\chiE$ of $E^{\times}$ via
$$\chi\begin{pmatrix}a&\sqrt{\epsilon}b\\ 0&\overline{a}^{-1}\end{pmatrix}=\chiE(a).$$
A \emph{principal\;series\;representation} of $G$ is the induced representation $\pi_{\chi}=\mathrm{Ind}_{B}^{G}\chi$, where $\chi$ is a character of $T$ extended trivially to $B$. 

\begin{remark}In the literature, the principal series representation associated to $\chi$ is often instead defined as the normalized induced representation $n  \scalebox{0.9}[0.9]{-}\mathrm{Ind}_{B}^{G}\chi=\pi_{ \delta^{\frac{1}{2}}\otimes \chi}$  where $\delta(t)=|\mathrm{det}(Ad(t)\mid_{\mathfrak{u}})|_{p}$ for $t\in T$, and $\mathfrak{u}=\mathrm{Lie}(U)$. Here $|\cdot|_{p}$ represents the $p$-adic norm.  In our setting, given $t(a)\in T$, we have $\delta(t)=|a\overline{a}|_p$. If $t(a)\in\mathcal{K}\cap T$, then $a\in\mathcal{O}_{E}^{\times}$ and so $\delta(t)=1$. We shall see that the branching rules of $\pi_{\chi}$ and $\pi_{\chi^{\prime}}$ agree whenever the restriction of $\chi$ and $\chi'$ to $T\cap\mathcal{K}$ agree. So in particular $n  \scalebox{0.9}[0.9]{-}\mathrm{Ind}_{B}^{G}\chi$ and $\mathrm{Ind}_{B}^{G}\chi$ have the same branching rules.
\end{remark}

Let $\chi$ be a character of $T$, and let $r \in \mathbb{R}_{\ge 0}$.  
We say that $\chi$ has \emph{depth} $r$ if  
\[
\chi|_{T_r} \ne \mathbbm{1}
\quad \text{and} \quad
\chi|_{T_{r+}} = \mathbbm{1}.
\]
As is standard, a character that is trivial on $T_0$ is said to have \emph{depth zero}.  
The \emph{true depth} of $\chi$ is defined as the depth of its restriction to $S$, where $S$ denotes a maximal $F$-split torus contained in $T$, as in~\eqref{definitionofS}.  
We say that $\chi$ has \emph{minimal depth} $r$ if the depth of $\chi$ and its true depth both equal $r$.

Note that restricting $\chi$ to $S$ corresponds to restricting $\chi^{\dagger}$ to $F^{\times}$. Since $T\cong E^{\times}$ and $S\cong F^{\times}$ whose filtrations are indexed by integers, it follows that both the depth and the true depth of characters of $T$ are integers. Moreover, for every $r\in\mathbb{Z}_{\geq 0}$, we have $S_{r}\subset T_{r}$, and hence the depth of a character of $T$ is always greater than or equal to its true depth.

\begin{lemma}\label{truedepth1}
	Let $m\in\mathbb{Z}_{\geq 0}$, and let $\chi\colon T\rightarrow \mathbb{C}^{\times}$ be a character of $T$ of true depth $m$. Then there exists a character $\phi\colon E^{1} \rightarrow \mathbb{C}^{\times}$ such that $\phi\circ \mathrm{det}\mid_{T_{m+1}}=\chi\mid_{T_{m+1}}$. Moreover, if $\chi\mid_{S_{0}}=\mathbbm{1}$, then $\chi\mid_{T}=\phi\circ \mathrm{det}\mid_{T}$ for some character $\phi$ of $E^{1}$.
\end{lemma}	
\begin{proof}
	For each $n\in\mathbb{Z}_{\geq 0}$, consider the map 	$\gamma_{n}\colon E^{\times}_{n}/F^{\times}_{n}\rightarrow E^{1}\cap E^{\times}_{n} $  given by $\gamma_{n}(x)=x\overline{x}^{-1}$. Then by Lang's theorem $\gamma_{n}$ is an isomorphism. Let $\chi$ be a character of $T$, that is trivial on $S_{n}$. We denote by $\widetilde{\phi}$ the character $(\chiE)\circ \gamma_{n}^{-1}\colon E^{1}\cap E^{\times}_{n} \rightarrow \mathbb{C}^{\times}$.
Let $x\in 	E^{\times}_{n}/ F^{\times}_{n} $.  We then have
	\begin{equation}\label{true depth m}\widetilde{\phi}(x\overline{x}^{-1})=\widetilde{\phi}(\gamma_{n}(x))=(\chiE\circ \gamma_{n}^{-1})(\gamma_{n}(x))=\chiE(x).\end{equation}
	Since $E^{1}$ is a compact abelian group and $\widetilde{\phi}$ is a character of an open subgroup of $E^{1}$, there exists an extension $\phi\colon E^{1}\rightarrow \mathbb{C}^{\times}$ of $\widetilde{\phi}$. As the determinant of matrices in $G$ lie in $E^{1}$, $\phi\circ \mathrm{det}$ is a well-defined character of $G$.
	Note that for any element $t(x)\in T$, $\mathrm{det}(t(x))=x\overline{x}^{-1}$, therefore taking $n=m+1$,  we obtain
	$\widetilde{\phi}\circ \mathrm{det}\mid_{T_{m+1}}=\chi\mid_{T_{m+1}}$, as was required.
	
	In particular,  if $\chi\mid_{S_{0}}=\mathbbm{1}$, then  taking $n=0$ there exists a character $\widetilde{\phi}$ of $E^{1}$, such that for all $x\in\mathcal{O}_{E}^{\times}/ \mathcal{O}_{F}^{\times}$, 
	\begin{equation}\label{512}
		\widetilde{\phi}(x\overline{x}^{-1})=\widetilde{\phi}(\gamma_{0}(x))=(\chiE\circ \gamma_{0}^{-1})(\gamma_{0}(x))=\chiE(x).\end{equation}
	Since $E$ over $F$ is an unramified extension, we have $E^{\times}/F^{\times}\cong \mathcal{O}_{E}^{\times}/\mathcal{O}_{F}^{\times}$, therefore~\eqref{512} holds for all $x\in E^{\times}/ F^{\times}$, and hence $\chi\circ\mathrm{det}\mid_{T}=\chi\mid_{T}$.	 
\end{proof}

\begin{theorem}\label{reduction-to-minimal-depth}
	Every character \( \chi \) of \( T \) of true depth $r$ is of the form $\chi = (\phi \circ \mathrm{det}) \otimes \chi',$ where \( \phi \) is a character of \( E^{1} \), and \( \chi' \) is a character of minimal depth $r$. 
\end{theorem}
\begin{proof}
First suppose that $\chi$ is a character of $T$ of true depth $r$ such that $\chi\mid_{S_{0}}\neq \mathbbm{1}$. Then $\chi\mid_{S_{r+1}}=\mathbbm{1}$. By Lemma~\ref{truedepth1}, there exists a character $\phi$ of $E^{1}$ such that 
	$$\chi\mid_{T{r+1}}=(\phi\circ\mathrm{det})\mid_{T_{r+1}}.$$
	Equivalently, $((\phi^{-1}\circ\mathrm{det})\otimes\chi)\mid_{T_{r+1}}=\mathbbm{1}.$ Since $(\phi^{-1}\circ\mathrm{det})$ is trivial on $S$, if $\chi\mid_{S_{0}}\neq \mathbbm{1}$,  then $$(\phi^{-1}\circ\mathrm{det})\otimes\chi\mid_{S_{r}}=\chi\mid_{S_r}\neq\mathbbm{1}.$$
	So, in particular $\chi$ is non-trivial on $T_{r}$. Thus  $((\phi^{-1}\circ\mathrm{det})\otimes\chi)$ has minimal depth $r$. Taking $\chi'=(\phi^{-1}\circ\mathrm{det})\otimes\chi$ we have
	$$\chi= (\phi\circ \mathrm{det})\otimes (\phi^{-1}\circ \mathrm{det})\otimes \chi=(\phi\circ \mathrm{det})\otimes\chi', $$ 
	as was required. Next suppose that $\chi\mid_{S_{0}}=\mathbbm{1}$. Then by Lemma~\ref{truedepth1} there exists a character $\phi\colon E^{1} \rightarrow \mathbb{C}^{\times}$ such that $\phi\circ \mathrm{det}\mid_{T}=\chi$, implying that $(\phi^{-1}\circ \mathrm{det})\otimes \chi=\mathbbm{1}$, and we may write 
	$\chi=(\phi\circ \mathrm{det})\otimes \mathbbm{1}$, and $\mathbbm{1}$ has minimal depth $0$.
\end{proof}	

\begin{corollary}
	Suppose the decomposition into irreducible representations of $\pi_{\chi}$ is known for all $\chi$ of minimal depth. Then so is the decomposition of $\pi_{(\phi\circ\mathrm{det})\otimes \chi}$ for every character $\phi$ of $E^{1}$, and thus for all principal series.
\end{corollary}	
\begin{proof}
	Let $\chi$ be a character of $T$ of minimal depth, and suppose
	$$\mathrm{Res}_{\mathcal{K}}\pi_{\chi}=\bigoplus_{i\in I}\sigma_i,$$
	where $\sigma_i$ are irreducible $\mathcal{K}$-representations and $I$ is some index set. Let $\phi$ be a character of $E^{1}$. By Lemma~\ref{principalseriescharacterreduction5}, multiplying an induced representation by a character of the group is equivalent to instead just multiplying the inducing character, therefore we have $\pi_{(\phi\circ\mathrm{det})\otimes \chi}\cong (\phi\circ\mathrm{det})\otimes \pi_{\chi}.$
	Hence, the restriction of $\pi_{(\phi\circ\mathrm{det})\otimes \chi}$	to $\mathcal{K}$ decomposes as
	$$\mathrm{Res}_{\mathcal{K}}\pi_{\chi}=\bigoplus_{i\in I}(\phi\circ\mathrm{det})|_{\mathcal{K}}\otimes \sigma_i,$$
	and each of these components is again irreducible.
	Since every character of \( T \) is of the form \( (\phi \circ \mathrm{det}) \otimes \chi' \), where \( \chi' \) has minimal depth (by Theorem~\ref{reduction-to-minimal-depth}), it follows that  the decomposition of all principal series representations is determined by that of those associated to characters of minimal depth.  
\end{proof}	
\begin{remark}\label{truedepth0trivialcharacter}
	If $\chi$ is a character of $T$ of true depth $0$ such that $\chi\mid_{S_{0}}=\mathbbm{1}$, then, by Lemma~\ref{truedepth1} there exists a character $\phi\colon E^{1} \rightarrow \mathbb{C}^{\times}$ such that $\phi\circ \mathrm{det}\mid_{T}=\chi$. By Lemma~\ref{principalseriescharacterreduction5} we have $\mathrm{Ind}_{B}^{G}\chi\cong (\phi\circ \mathrm{det})\otimes \mathrm{Ind}_{B}^{G}\mathbbm{1}$.	Thus in this case we can always choose $\mathbbm{1}$ as the representative and provide branching rules for $\mathrm{Ind}_{B}^{G}\mathbbm{1}$.
\end{remark}

We conclude this section with another reduction lemma that will be useful in~\S\ref{restrictiontoK2rprincipalseries}.
\begin{lemma}\label{centralcharacterofprincipalserieshasdepth-zero}
	Let $\chi$ be a character of $T$ of minimal depth $r \ge 0$, and let $\pi_{\chi} = \mathrm{Ind}_{B}^{G}\chi$ be the corresponding principal series representation of $G$. 
	Then there exist a character $\phi$ of $E^{1}$ and a character $\chi_{0}$ of $T$ such that
	\[
	\chi = (\phi\circ \mathrm{det}) \otimes \chi_{0}, 
	\qquad 
	\pi_{\chi} \;\cong\;  (\phi\circ \mathrm{det}) \otimes \pi_{\chi_{0}},
	\]
	and moreover the central character $\theta$ of $\pi_{\chi_{0}}$ equals $\delta^{k}$ for some $k \in \{0,1\}$, where $\delta$ denotes the non-trivial quadratic character of $E^{1}$.
\end{lemma}

\begin{proof}
	Let $\theta = \mathrm{Res}_{Z}\chi$ denote the restriction of $\chi$ to the center $Z$. As noted in~\S\ref{the group}, the center $Z$ of $G$ can be identified with 
	$E^{1}$, the group of norm-one elements of $E$ over $F$.  The group $E^{1}/ (E^{1})^{2}$ is size $2$. Let $\delta$ denote the non-trivial quadratic character of $E^{1}$. Then $\delta$ has depth-zero and every character $\theta$ of $E^{1}$ can be written as either $\theta=\phi^{2}$ or $\theta=\delta \phi^{2}$ for some character $\phi$ of $E^{1}$. Define  $\chi_{0} := (\phi^{-1} \circ \mathrm{det})\otimes\chi  ;$ then $\mathrm{Res}_{Z}\chi_{0} = \delta^{k}$. We then have $\chi = (\phi\circ \mathrm{det}) \otimes \chi_{0}$. 
	Applying Lemma~\ref{principalseriescharacterreduction5}, the lemma follows.
\end{proof}

\section{A canonical decomposition upon restriction to $\mathcal{K}$}\label{a canonical decomposition}

Let $\chi$ be a character of $T$ of minimal depth $r$, and let  $(\pi_{\chi}, V_{\chi})$ be the associated principal series representation.  In this section, we describe the decomposition of the restriction $\pi_{\chi}|_{\mathcal{K}}$ into irreducible $\mathcal{K}$-representations.

Since $G=\mathcal{K}B$, applying Mackey theory we obtain
\begin{equation}
	\mathrm{Res}_{\mathcal{K}}^{G}\mathrm{Ind}_{B}^{G}\chi=\mathrm{Ind}_{B\cap\mathcal{K}}^{\mathcal{K}}\chi.
\end{equation}

Since irreducible representations of compact groups are finite dimensional, we note that $V_{\chi}$ is highly reducible as a representation of $\mathcal{K}$. Thus we need to decompose $V_{\chi}$ further into irreducible $\mathcal{K}$- representations.
As discussed in~\S\ref{representation theory},  the decomposition of  $V_{\chi}$ reduces to decomposing each $V_{\chi}^{\mathcal{K}_n}$ for all $n\geq 0$.

\begin{lemma}\label{branch1}
If $\chi$ has depth $r$ then so does $\pi_{\chi}$. In particular, if \( \chi \neq \mathbbm{1} \) then
\[
V_{\chi}^{\mathcal{K}_n} = \{0\} \quad \text{for all } n \leq r.
\]

\end{lemma}
\begin{proof}
Let \( n \leq r \). We want to show that \( V_{\chi}^{\mathcal{K}_n} = \{0\} \). Assume on the contrary that there exists \( f \neq 0 \in V_{\chi}^{\mathcal{K}_n} \). Then there exists \( x \in \mathcal{K} \) such that \( f(x) \neq 0 \). Since \( \chi \) has depth \( r \), there exists \( b \in B \cap \mathcal{K}_n \) such that \( \chi(b) \neq 1 \). As \( \mathcal{K}_n \trianglelefteq \mathcal{K} \), the conjugate \( t := x^{-1} b x \in \mathcal{K}_n \), and it acts trivially on $f$. This gives
\[
f(x) = \pi_{\chi}(t)f(x) = f(x t) = f(b x) = \chi(b) f(x),
\]
which is impossible since \( \chi(b) \neq 1 \) and \( f(x) \neq 0 \). Therefore, \( V_{\chi}^{\mathcal{K}_n} = \{0\} \) when \( n \leq r \).

\end{proof}	 
Observe that if $n$ is greater than the depth of $\chi$, then $\chi$ acts trivially on $B\cap\mathcal{K}_{n}$. Thus the restriction of $\chi$ to $B\cap\mathcal{K}$ extends trivially to a character of $(B\cap\mathcal{K})\mathcal{K}_{n}$. As in \S\ref{the group}, for each $n\geq 0$, let us denote $B\cap\mathcal{K}_{n}$ by $\mathcal{B}_{n}$. Since $\mathcal{K}_{0}=\mathcal{K}$ we simply denote $\mathcal{B}_{0}$ by $\mathcal{B}$.

\begin{lemma}\label{principalseriescharacterreduction6}
Let $r\in\mathbb{Z}_{\geq 0}$, and let \( \chi  \) be a character of \( T \) of depth $r$. Then for any \( n \geq r+1\), we have
\[
V_{\chi}^{\mathcal{K}_{n}} \cong \mathrm{Ind}_{\mathcal{B}\mathcal{K}_{n}}^{\mathcal{K}} \chi.
\]	
\end{lemma}	
\begin{proof}
We show that the vectors spaces on which the two representation are acting are in fact the same.  We have
\[
W_{1}
:=V_{\chi}^{\mathcal{K}_{n}}
=\Bigl\{
f:\mathcal{K}\to\mathbb{C}
\,\Bigm|\,
f(bk)=\chi(b)f(k)\;\forall b\in B\cap\mathcal{K},\;\mathrm{and}\;
f(kk_{n})=f(k)\;\forall k_{n}\in\mathcal{K}_{n}
\Bigr\}.
\]
and set
\[
W_{2}
:=\mathrm{Ind}_{\mathcal{B}\mathcal{K}_{n}}^{\mathcal{K}} \chi
=\Bigl\{
\ell:\mathcal{K}\to\mathbb{C}
\,\Bigm|\,
\ell(bk)=\chi(b)\ell(k),\;\mathrm{and}\;\ell(k_{n}k)=\ell(k)\;
\;\forall b\in \mathcal{B},\;\forall k_{n}\in\mathcal{K}_{n}\; \forall k\in\mathcal{K}
\Bigr\}.
\]	
Note that $W_{1}$ consists of function $f\colon \mathcal{K}\rightarrow \mathbb{C}$ that are constant on the left cosets of $\mathcal{K}_{n}$. Since $\mathcal{K}_{n}$ is a normal subgroup of $\mathcal{K}$, the left and the right cosets agree. Therefore, the spaces \( W_{1} \) and \( W_{2} \) are equal, and hence the corresponding representations are isomorphic.
\end{proof}

\begin{lemma}\label{dimensionofVchiKn}
Let $r\in\mathbb{Z}_{\geq 0}$, and let \( \chi  \) be a character of \( T \) of depth $r$. Then for any \( n \geq r+1\), the depth of every irreducible component of $V_{\chi}^{\mathcal{K}_{n}}$ is less than $n$, and the degree of $V_{\chi}^{\mathcal{K}_{n}}$ equals $q^{n-1}(q+1)$.
\end{lemma}	
\begin{proof}
Let $n\geq r+1$, by definition of $V_{\chi}^{\mathcal{K}_{n}}$, $\mathcal{K}_{n}$ acts trivially on the entire representation, hence on every irreducible component. Hence the maximal depth of any irreducible component of $V_{\chi}^{\mathcal{K}_{n}}$ must be strictly less that $n$.

The degree of $V_{\chi}^{\mathcal{K}_{n}}$ is given by the index $[\mathcal{K}\colon \mathcal{B}\mathcal{K}_{n}].$ Since $\mathcal{K}_{n}\trianglelefteq \mathcal{K}$, and $\mathcal{K}_{n}\trianglelefteq \mathcal{B}\mathcal{K}_{n}$, applying group isomorphism theorems the above index can be computed as
$$[\mathcal{K}\colon \mathcal{B}\mathcal{K}_{n}]=\frac{[\mathcal{K}\colon\mathcal{K}_{1}][\mathcal{K}_{1}\colon\mathcal{K}_{n}]}{[\mathcal{B}\colon \mathcal{B}_{1}][\mathcal{B}_{1}\colon \mathcal{B}_{n}]}.
$$
Since $\mathcal{K}/ \mathcal{K}_{1}\cong \mathbb{U}(1,1)(\mathfrak{f})$ and $\mathcal{B}/\mathcal{B}_1\cong \mathbb{B}(\mathfrak{f})$, the index $[\mathcal{K}\colon \mathcal{K}_{1}]=|\mathbb{U}(1,1)(\mathfrak{f}|=q(q+1)(q-1)^{2}$, and the index $[\mathcal{B}\colon \mathcal{B}_{1}]=|\mathbb{B}(\mathfrak{f})|=q(q^2-1)$ respectively. The remaining indices  $[\mathcal{K}_{1}\colon\mathcal{K}_{n}], [\mathcal{B}_{1}\colon \mathcal{B}_n] $ are same as the indices of corresponding $\mathcal{O}_{F}$-modules in the Lie algebra. Thus we have
$$[\mathcal{K}\colon \mathcal{B}\mathcal{K}_{n}]=\frac{[\mathcal{K}\colon\mathcal{K}_{1}][\mathcal{K}_{1}\colon\mathcal{K}_{n}]}{[\mathcal{B}\colon \mathcal{B}_{1}][\mathcal{B}_{1}\colon \mathcal{B}_{n}]}=\frac{q(q^2-1)(q+1)q^{4(n-1)}}{q(q^2-1)q^{3(n-1)}}=(q+1)q^{n-1}.$$
\end{proof}

We now come to the major technical result of this section, which will serve as the key ingredient in establishing a canonical decomposition upon restriction to $\mathcal{K}$.
\begin{proposition}\label{dimension1principalseries} Let $\chi$ be a character of $T$ of minimal depth $r$. Then for $d\geq r+1$ we have
$$\mathrm{dim}_{\mathbb{C}}(\mathrm{Hom}_{\mathcal{K}}(V_{\chi}^{\mathcal{K}_{d}}, V_{\chi}^{\mathcal{K}_{d}}))=\begin{cases} d+1 &\mathrm{if}\;r=0\;\mathrm{and}\;\chi\mid_{S_{0}}= \mathbbm{1},\\
	d-r&\mathrm{otherwise.}\end{cases}$$
	\end{proposition}
	\begin{proof}
Let $d\geq r+1$. By Lemma~\ref{principalseriescharacterreduction6} we have $\mathrm{Ind}_{\mathcal{B}\mathcal{K}_{d}}^{\mathcal{K}} \chi\cong V_{\chi}^{\mathcal{K}_{d}}$. Thus
$$\mathrm{Hom}_{\mathcal{K}}( V_{\chi}^{\mathcal{K}_{d}},  V_{\chi}^{\mathcal{K}_{d}})\cong\mathrm{Hom}_{\mathcal{K}}(\mathrm{Ind}_{\mathcal{B}\mathcal{K}_{d}}^{\mathcal{K}} \chi, \mathrm{Ind}_{\mathcal{B}\mathcal{K}_{d}}^{\mathcal{K}} \chi).$$
Using Frobenius reciprocity and Mackey theory we have 
\begin{align*}\mathrm{Hom}_{\mathcal{K}}(\mathrm{Ind}_{\mathcal{B}\mathcal{K}_{d}}^{\mathcal{K}} \chi, \mathrm{Ind}_{\mathcal{B}\mathcal{K}_{d}}^{\mathcal{K}} \chi)&=\bigoplus_{g\in\mathcal{B}\mathcal{K}_{d}\backslash \mathcal{K}/ \mathcal{B}\mathcal{K}_{d}}\mathrm{Hom}_{(\mathcal{B}\mathcal{K}_{d})\cap (\mathcal{B}\mathcal{K}_{d})^g}(\chi, \chi^{g}).\end{align*}	
From Proposition~\ref{double coset for principal series} it follows that a set of double coset representatives for $(\mathcal{B}\mathcal{K}_{d})\backslash \mathcal{K}/ (\mathcal{B}\mathcal{K}_{d}) $ is given by
\begin{equation}
	P:=	\left\{\mathrm{id}:=\begin{pmatrix}1&0\\0&1\end{pmatrix},\;\mathrm{w}:=\begin{pmatrix}0&1\\1&0\end{pmatrix},\;\begin{pmatrix}1&0\\ \sqrt{\epsilon}\varpi^{k}&1\end{pmatrix}\;\middle|\; 1\leq k<d\right\}.
\end{equation}	
Thus, the dimension of the homomorphism space  $\mathrm{Hom}_{\mathcal{K}}(V_{\chi}^{\mathcal{K}_{d}}, V_{\chi}^{\mathcal{K}_{d}})$ is equal to the number of elements $g\in P$ such that 
\begin{equation}\chi(h)=\chi^{g}(h)\;\hspace{2em}\;\mathrm{for\;all}\; h\in(\mathcal{B}\mathcal{K}_{d})\cap (\mathcal{B}\mathcal{K}_{d})^g. \end{equation}
Let  $bk_{d}\in \mathcal{B}\mathcal{K}_{d}$, where $b=\begin{pmatrix}a&u\\ 0&\overline{a}^{-1}\end{pmatrix}\in\mathcal{B}$, and $k_{d}\in\mathcal{K}_{d}$. Since $\mathcal{K}_{d}\trianglelefteq\mathcal{K}$,  we have $g^{-1}k_{d}g\in\mathcal{K}_{d}$ for all $g\in P$. Therefore, for all $g\in P$,  $bk_{d}\in (\mathcal{B}\mathcal{K}_{d})^{g}$ if and only if $g^{-1}bg\in \mathcal{B}\mathcal{K}_{d}$. Moreover,
$$\chi^{g}(bk_{d})=\chi(g^{-1}bgg^{-1}k_dg)=\chi(g^{-1}bg)\chi(g^{-1}k_{d}g)=\chi(g^{-1}bg).$$
Thus, the dimension of homomorphism space $\mathrm{Hom}_{\mathcal{K}}(V_{\chi}^{\mathcal{K}_{d}}, V_{\chi}^{\mathcal{K}_{d}})$ reduces to finding  the number of elements $g\in P$ such that 
\begin{equation}\label{chiintertwine}\chi(b)=\chi^{g}(b)\;\hspace{2em}\;\mathrm{for\;all}\; b\in(\mathcal{B}\cap (\mathcal{B}\mathcal{K}_{d})^g). \end{equation}
Clearly $g=\mathrm{id}$ satisfies~\eqref{chiintertwine}. Now let $g=\mathrm{w}$. We compute
\begin{align*}
	\mathrm{w}^{-1}b\mathrm{w}=\begin{pmatrix}\overline{a}^{-1}&0\\ u&a\end{pmatrix}=\begin{pmatrix}\overline{a}^{-1}&0\\ 0&a
	\end{pmatrix}\begin{pmatrix}1&0\\a^{-1}u&1\end{pmatrix}.
\end{align*}
Thus $\mathrm{w}^{-1}b\mathrm{w}\in \mathcal{B}\mathcal{K}_{d}$, for all $a\in\mathcal{O}_{E}^{\times}$, and for all $u\in\mathfrak{p}_{E}^{d}$ such that $\overline{a}u\in \sqrt{\epsilon}F$.

Since $\chi$ is trivial on $\mathcal{K}_{d}$, we then have   $$\chi^{\mathrm{w}}(b)=\chi(\mathrm{w}^{-1}b\mathrm{w})=\chiE(\overline{a}^{-1}).$$

For the characters to agree we must have $\chiE(a)=\chiE(\overline{a}^{-1})$ for all $a\in\mathcal{O}_{E}^{\times}$, which implies that $\chiE(a\overline{a})=1$ for all $a\in\mathcal{O}_{E}^{\times}$ which happen only when $\chiE\mid_{\mathcal{O}_{F}^{\times}}= \mathbbm{1}$ since the norm map is surjection on $\mathcal{O}_{F}^{\times}$. Thus $\mathrm{w}$ satisfies~\eqref{chiintertwine} only if $\chi\mid_{S_{0}}=\mathbbm{1}$.

Lastly, let $g_k=\begin{pmatrix}1&0\\ \sqrt{\epsilon}\varpi^{k}&1\end{pmatrix}$. Then 
\begin{align*}
	g_{k}^{-1}bg_{k}&=\begin{pmatrix}a+u\sqrt{\epsilon}\varpi^{k}&u\\ (-a+\overline{a}^{-1})\sqrt{\epsilon}\varpi^{k}-u\epsilon\varpi^{2k}& \overline{a}^{-1}-u\sqrt{\epsilon}\varpi^{k}\end{pmatrix}.
\end{align*}	
If $g_{k}^{-1}bg_{k}\in \mathcal{B}\mathcal{K}_{d}$, then there exists elements $t:=\begin{pmatrix}x&y\\ 0&\overline{x}^{-1}\end{pmatrix}\in \mathcal{B}$, and $k_{d}:=\begin{pmatrix}1+c\varpi^{d}& e\varpi^{d}\\ f\varpi^{d}& 1+l\varpi^{d}\end{pmatrix}\in\mathcal{K}_{d}$ with $x\in\mathcal{O}_{E}^{\times}$, $y, c, e, f, l\in\mathcal{O}_{E}$ such that 
\begin{equation}\label{eqtosolve2}g_{k}^{-1}bg_{k}=tk_{d}.\end{equation} 
Solving~\eqref{eqtosolve2} for $x$ yields $x=a(1+c\varpi^{d})^{-1}(1+a^{-1}\sqrt{\epsilon}\varpi^{k}-a^{-1}fy\varpi^{d})$. If $k>r$, then $(1+a^{-1}\sqrt{\epsilon}\varpi^{k}-a^{-1}fy\varpi^{d}), (1+c\varpi^{d})^{-1}\in\mathfrak{p}_{E}^{r+1}$. Since $\chi$ has depth $r$, we have
$$\chi^{g_{k}}(b)=\chiE(x)=\chiE(a)\chiE((1+c\varpi^{d})^{-1})\chiE(1+a^{-1}\sqrt{\epsilon}\varpi^{k}-a^{-1}fy\varpi^{d})=\chiE(a)=\chi(b).$$

Hence  for $k>r$, $g_{k}$ satisfies~\eqref{chiintertwine} on the given intersection. 

We claim that for $1 \leq k \leq r$, the element $g_k$ does not intertwine $\chi$. Since $k \leq r$ and $\chi$ has true depth $r$, we have $\chiE\mid_{1+\mathfrak{p}_F^{k}} \neq 1$. This key observation allows us to find elements $h \in \mathcal{B}\mathcal{K}_{d} \cap (\mathcal{B}\mathcal{K}_{d})^{g_k}$ such that $\chi(h) \neq \chi^{g_k}(h)$, thereby proving the claim.

Given any $a = 1 + a_0 \varpi^k \in 1 + \mathfrak{p}_F^k$, we set $u := \frac{a^{-1} - a}{\sqrt{\epsilon} \varpi^k} \in \sqrt{\epsilon} \mathcal{O}_F,$ and we define 
\[
h := \begin{pmatrix} a & u \\ 0 & a^{-1} \end{pmatrix}\in\mathcal{B}.
\]
Then the (2, 1) entry of $g_{k}^{-1}hg_{k}$ equals $0$, and hence $h$ lies in $\mathcal{B}\mathcal{K}_d \cap (\mathcal{B}\mathcal{K}_d)^{g_k}$. We also have
$$g_{k}^{-1}hg_{k}=\begin{pmatrix}a+u\sqrt{\epsilon}\varpi^{k}&u\\ 0& a^{-1}-u\sqrt{\epsilon}\varpi^{k}\end{pmatrix}=\begin{pmatrix}a&0\\0&a^{-1}\end{pmatrix}\begin{pmatrix}1+a^{-1}u\sqrt{\epsilon}\varpi^{k}&a^{-1}u\\ 0&1-au\sqrt{\epsilon}\varpi^{k}\end{pmatrix}.$$

We compute
\[
\chi^{g_k}(h) = \chi(g_{k}^{-1}hg_{k})=\chiE(a)\chiE(1 + a^{-1} u \sqrt{\epsilon} \varpi^k).
\]
As $a$ ranges over all elements of $1 + \mathfrak{p}_F^k$, the term $1 + a^{-1} u \sqrt{\epsilon} \varpi^k$ ranges over all of $1 + \mathfrak{p}_F^k$ as well. Therefore, we can choose $a$ such that $\chiE(1 + a^{-1} u \sqrt{\epsilon} \varpi^k) \neq 1$, and hence $\chi(h) \neq \chi^{g_k}(h)$ as required. 

Therefore, if $\chi|_{S_{0}} = \mathbbm{1}$, then all elements of $P$ satisfy~\eqref{chiintertwine}, therefore the dimension equals the cardinality of $P$ which is $d+1$. On the other hand if $\chi\mid_{S_{0}}\neq \mathbbm{1}$, then the elements of $P$ that satisfy~\eqref{chiintertwine} are $\mathrm{id}$, and $\{g_{k}\;|\;r<k<d\}$, which counts to $1+d-r-1=d-r$. This proves the proposition.
\end{proof}	

As noted earlier,  the decomposition of \(V_{\chi}\) reduces to determining the decomposition of each \(V_{\chi}^{\mathcal{K}_n}\) for all \(n \geq 0\). 
Observe that if \(\chi\) is the trivial character, then by the preceding proposition,
\[
\dim_{\mathbb{C}}\!\left(\mathrm{Hom}_{\mathcal{K}}\!\left(V_{\chi}^{\mathcal{K}_{1}},\, V_{\chi}^{\mathcal{K}_{1}}\right)\right)=2.
\]
In contrast, if \(\chi\) is a character of \(T\) of true depth \(r \geq 0\) such that \(\chi|_{S_{0}} \neq \mathbbm{1}\), then
\[
\dim_{\mathbb{C}}\!\left(\mathrm{Hom}_{\mathcal{K}}\!\left(V_{\chi}^{\mathcal{K}_{r+1}},\, V_{\chi}^{\mathcal{K}_{r+1}}\right)\right)=1.
\]
Hence in the latter case \(V_{\chi}^{\mathcal{K}_{r+1}}\) is irreducible, whereas the following lemma describes the former case.

\begin{lemma}\label{principalseries2}
If $\chi = \mathbbm{1}$ is the trivial character, then
\[
V_{\chi}^{\mathcal{K}_{1}} \cong \mathbbm{1} \oplus \mathrm{St},
\]
where $\mathbbm{1}$ and $\mathrm{St}$ denote the inflations to $\mathcal{K}$ of the trivial and Steinberg representations $\mathbbm{1}_{q}$ and $\mathrm{St}_{q}$ of $\mathrm{U}(1,1)(\mathfrak{f})$, respectively.
\end{lemma}

\begin{proof}
Recall that the Steinberg representation $\mathrm{St}_{q}$ of $\mathrm{U}(1,1)(\mathfrak{f})$ is defined by the property that $\mathrm{Ind}_{B(\mathfrak{f})}^{\mathrm{U(1,1)}(\mathfrak{f})}\mathbbm{1}\cong\mathbbm{1}_{q}\oplus \mathrm{St}_{q}$, where $\mathbbm{1}_{q}$ denotes the trivial representation of $\mathrm{U(1,1)}(\mathfrak{f})$, and the degree of $\mathrm{St}_{q}$ is $q$. Let $\chi=\mathbbm{1}$ be the trivial character of $T$. We then have
$$V_{\chi}^{\mathcal{K}_{1}}\cong \mathrm{Ind}_{\mathcal{B}\mathcal{K}_{1}}^{\mathcal{K}}\mathbbm{1}\cong \mathrm{Ind}_{\mathcal{B}/ \mathcal{B}_{1}}^{\mathcal{K}/\mathcal{K}_{1}}\mathbbm{1}\cong \mathrm{Ind}_{B(\mathfrak{f})}^{\mathrm{U(1,1)}(\mathfrak{f})}\mathbbm{1}_{q}\cong\mathbbm{1}_{q}\oplus \mathrm{St}_{q} $$
where we use the identification $\mathcal{K}/\mathcal{K}_{1} \cong \mathrm{U}(1,1)(\mathfrak{f})$.
Inflating $\mathbbm{1}_{q}$ and $\mathrm{St}_{q}$ to $\mathcal{K}$ (via the projection $\mathcal{K} \to \mathcal{K}/\mathcal{K}_1$) gives representations $\mathbbm{1}$ and $\mathrm{St}$ of $\mathcal{K}$, and the lemma follows. 
\end{proof}

We now state the main theorem of this section which provides a canonical decomposition of principal series representations of $G$ upon restriction to $\mathcal{K}$.

\begin{theorem}\label{thm2}
Let $\chi$ be a character of $T$ of minimal depth $r$, and let $(\pi_{\chi}, V_{\chi})$ denote the principal series representation associated to $\chi$. Then there exists irreducible representations $\mathcal{W}_{d, \chi}$ of degree $(q^{2}-1)q^{d-1}$, for $d\geq r+1$ such that
$$\mathrm{Res}_{\mathcal{K}}\pi_{\chi}\cong V_{\chi}^{\mathcal{K}_{r+1}}\oplus\bigoplus_{d \geq r+1}\mathcal{W}_{d,\chi}$$
where $V_{\chi}^{\mathcal{K}_{1}}=\mathbbm{1}_{q}+\mathrm{St}_{q}$ if $\chi=\mathbbm{1}$, and  $V_{\chi}^{\mathcal{K}_{r+1}}$ is an irreducible representation of degree $(q+1)q^{r}$ otherwise.
\end{theorem}
\begin{proof}
If $\chi=\mathbbm{1}$, then  $V_{\chi}^{\mathcal{K}_{1}}=\mathbbm{1}_{q}+\mathrm{St}_{q}$ by Lemma~\ref{principalseries2}. If $\chi\neq \mathbbm{1}$, then from Lemma~\ref{dimension1principalseries} we have that
$\mathrm{dim}_{\mathbb{C}}(\mathrm{Hom}_{\mathcal{K}^{r+1}}(V_{\chi}^{\mathcal{K}_{r+1}}, V_{\chi}^{\mathcal{K}_{r+1}}))=1$, hence $V_{\chi}^{\mathcal{K}_{r+1}}$ is irreducible. Since $\chi$ is one- dimensional the degree is given by the index $[\mathcal{K}\colon \mathcal{B}\mathcal{K}_{r+1}]=(q+1)q^{r}$ by Lemma~\ref{dimensionofVchiKn}. 

Now let $d\geq r+1$. Since $V_{\chi}^{\mathcal{K}_{d}}\subsetneq V_{\chi}^{\mathcal{K}_{d+1}}$, by Maschke's theorem there exists a representation $\mathcal{W}_{d, \chi}$ of $\mathcal{K}$ such that $V_{\chi}^{\mathcal{K}_{d+1}}= V_{\chi}^{\mathcal{K}_{d}}\oplus \mathcal{W}_{d, \chi}$. We therefore have
\begin{align*}
	\mathrm{dim}_{\mathbb{C}}(\mathrm{Hom}_{\mathcal{K}}(V_{\chi}^{\mathcal{K}_{d+1}}, V_{\chi}^{\mathcal{K}_{d+1}}))=&\mathrm{dim}_{\mathbb{C}}(\mathrm{Hom}_{\mathcal{K}}(V_{\chi}^{\mathcal{K}_{d}}, V_{\chi}^{\mathcal{K}_{d}}))+2\mathrm{dim}_{\mathbb{C}}(\mathrm{Hom}_{\mathcal{K}^{n}}(V_{\chi}^{\mathcal{K}_{d}}, \mathcal{W}_{d, \chi}))\\
	&+\mathrm{dim}_{\mathbb{C}}(\mathrm{Hom}_{\mathcal{K}}(\mathcal{W}_{d, \chi}, \mathcal{W}_{d, \chi})).
\end{align*}
By Lemma~\ref{dimension1principalseries} the first term is $d+1-r$, and  the second term is $d-r$. Since the last term is at least one, it follows that the cross terms are zero and $\mathcal{W}_{d, \chi}$ is irreducible. Since the cross terms are zero, it follows that $\mathcal{W}_{d, \chi}$ has zero $\mathcal{K}_{d}$ invariants, and hence it has depth $d$.

We now prove by induction on $d$ that for all $d\geq r+1$, 
$V_{\chi}^{\mathcal{K}_d}=\bigoplus_{i=r}^{d-1}\mathcal{W}_{i, \chi}$. For the base case, we denote $V_{\chi}^{\mathcal{K}_{r+1}}$ by $\mathcal{W}_{r, \chi}$. Let us assume that the result holds for some $d\geq r+1$, then we want to show that it holds for $d+1$. We have already shown that
$V_{\chi}^{\mathcal{K}_{d+1}}=V_{\chi}^{\mathcal{K}_{d}}\oplus\mathcal{W}_{d, \chi}$. By the induction hypothesis we have $V_{\chi}^{\mathcal{K}_{d}}=\bigoplus_{i=r}^{d-1}\mathcal{W}_{i, \chi}$; hence the claim. 
Finally, we have
$$V_{\chi}=\bigcup_{n\geq 0}V_{\chi}^{\mathcal{K}_n}=\bigcup_{n\geq r+1}\left(\bigoplus_{i=r}^{n-1}\mathcal{W}_{i, \chi}\right)=\bigoplus_{d\geq r}\mathcal{W}_{d, \chi}.$$

The degree of $\mathcal{W}_{d, \chi}$ is given by $\mathrm{dim}_{\mathbb{C}}(V_{\chi}^{\mathcal{K}_{d+1}})-\mathrm{dim}_{\mathbb{C}}(V_{\chi}^{\mathcal{K}_{d}})=q^{d+1}+q^{d}-q^{d}-q^{d-1}=(q^2-1)q^{d-1}$.

\end{proof}

Since the representations \( \mathcal{W}_{d,\chi} \) have distinct degrees for distinct values of \( d \), we obtain the following immediate corollary.

\begin{corollary}
	The decomposition in Theorem~\ref{thm2} is multiplicity-free.
\end{corollary}

\section{Positive depth representations of $\mathcal{K}$}\label{chapter3}
In Section~\ref{a canonical decomposition}, we provided a canonical decomposition of $\pi_{\chi}$ upon restriction to $\mathcal{K}$. However, the components $\mathcal{W}_{d,\chi}$ were obtained as \emph{quotients} of certain induced representations. In this section, we take a brief detour to construct explicit irreducible representations of $\mathcal{K}$, and in the subsequent section we show that the components $\mathcal{W}_{d,\chi}$ are in fact isomorphic to the representations constructed here. This identification yields a more concrete and explicit description of the decomposition.

Shalika, in his thesis~\cite{Sha2004}, constructed all irreducible representations of $\mathrm{SL}(2,\mathcal{O}_{F})$ for $p \neq 2$. Here, we extend his method to construct certain positive-depth irreducible representations of the compact open subgroup $\mathcal{K}$ that appear as non-types in the branching rules to follow. Since $\mathbb{SU}(1,1)(\mathcal{O}_{F})$ is isomorphic to $\mathrm{SL}(2,\mathcal{O}_{F})$, the construction proceeds analogously, although it has not been explicitly carried out in the literature. In addition, we construct representations associated with nilpotent orbits of elements in the Lie algebra of $G$, using the same approach. These representations play a crucial role in describing the branching rules in a neighborhood of the identity element.

\subsection{Key concepts involved in the construction}

We begin by recalling a foundational result from Clifford theory, as given in \cite[Lemmas~4.1.1 and~4.1.3]{Sha2004}, which serves as a key tool for extending irreducible representations from a normal subgroup to the whole group.

\begin{theorem}[Clifford Theory {\cite[Lemmas~4.1.1 and~4.1.3]{Sha2004}}]\label{Cliff1}
	Let $K$ be a finite group, and let $N \triangleleft K$ be a normal subgroup.
	\begin{enumerate}
		\item[(a)]\label{Cliff1a} Let $\phi$ be an irreducible representation of $N$, and define  
		\[
		N_{K}(\phi) = \{k \in K \mid \phi^{k} \cong \phi \}
		\]
		to be its normalizer in $K$. If $\sigma_{i}$ is an irreducible representation of $N_{K}(\phi)$ such that 
		$\phi$ occurs as a subrepresentation of $\sigma_i\!\mid_{N}$, then the induced representation 
		$\mathrm{Ind}_{N_{K}(\phi)}^{K}\sigma_i$ is irreducible. Moreover,  
		\[
		\mathrm{Ind}_{N_{K}(\phi)}^{K}\sigma_i \;\cong\; \mathrm{Ind}_{N_{K}(\phi)}^{K}\sigma_j 
		\quad \text{if and only if} \quad
		\sigma_i \cong \sigma_j.
		\]
		\item[(b)]\label{Cliff1b} Suppose $K = AN$ with $A \leq K$ a subgroup and $N \triangleleft K$.  
		Let $\Psi$ be a one-dimensional representation of $N$ that is $K$-invariant, i.e.\ $N_{K}(\Psi)=K$. Then
		\begin{enumerate}
			\item The irreducible representations $\sigma$ of $K$ satisfying $\sigma\!\mid_{N} \supset \Psi$ are in one-to-one correspondence with the irreducible representations $\zeta$ of $A$ satisfying $\zeta\!\mid_{A \cap N} \supset \Psi\!\mid_{A \cap N}$, 
			where $\sigma(an)=\zeta(a)\Psi(n)$	for all $a\in A$ and $n\in N$.
			
			\item For each $\sigma$ as in $(a)$, the multiplicity of $\sigma$ in $\mathrm{Ind}_{N}^{K}\Psi$ is equal to the degree of $\sigma$.
		\end{enumerate}
	\end{enumerate}
\end{theorem}

\subsection{Positive-depth representations of $\mathcal{K}$}
In this section, we construct certain positive-depth representations of $\mathcal{K}$  using Theorem~\ref{Cliff1}.

Recall that in~\S\ref{the field} we fixed an additive character $\psi$ of $E$ that is trivial on $\mathfrak{p}_{E}$ but nontrivial on $\mathcal{O}_{E}$.
 The Moy--Prasad isomorphism for $\mathcal{K}$ is given by
\begin{align*}
	\mathcal{K}_{\frac{d}{2}+}/\mathcal{K}_{d+}
	&\cong\mathfrak{k}_{ \frac{d}{2}+}/\mathfrak{k}_{d+}\\
	k+\mathcal{K}_{d+}&\mapsto (k-I)+\mathfrak{k}_{d+}
\end{align*}	
Let $X\in \mathfrak{k}_{-d}$. Then the function 
\begin{equation}\label{characterdefinition}
	k \mapsto \Psi_X(k) = \psi(\mathrm{Tr}(X(k-I)))
\end{equation} 
defines a character of the group $\mathcal{K}_{\frac{d}{2}+}$ of depth $d$, and all such character arise in this way. We focus on elements $X\in\mathfrak{k}_{-d}$ of the form
\begin{equation}\label{element X}
	X=X(z)+\widetilde{X}(u, v)=\begin{pmatrix}z\sqrt{\epsilon}& 0\\ 0&z \sqrt{\epsilon}\end{pmatrix}+\begin{pmatrix}0&u\sqrt{\epsilon}\\
	v\sqrt{\epsilon}&0\end{pmatrix},\end{equation}
where $z, u, v\in F$ such that $\nu(z)\geq-d$ and $\nu(v)>\nu(u)=-d$. Note that $z, v$ and $u$ are uniquely determined by the character $\Psi_{X}$ only modulo $\mathfrak{p}_{F}^{-\lceil\frac{d}{2}\rceil}$. Let $C_{\mathcal{K}}\left(X+\mathfrak{k}_{-\frac{d}{2}}\right)$ be the centralizer of the coset $X+\mathfrak{k}_{-\frac{d}{2}}$ in $\mathcal{K}$.

Thus, we are precisely in the setting of Theorem~\ref{Cliff1}(a). We have a normal subgroup $\mathcal{K}_{\frac{d}{2}+}$ of $\mathcal{K}$ and an irreducible representation of it given by $\Psi_X$.  
To produce an irreducible representation of $\mathcal{K}$, we need an irreducible representation of the normalizer of $\Psi_X$ in $\mathcal{K}$ whose restriction to $\mathcal{K}_{\frac{d}{2}+}$ contains $\Psi_X$. Therefore, we now compute the normalizer of $\Psi_X$ in $\mathcal{K}$.

\begin{lemma}\label{normalizerequalscentralizerofacoset}
	We have $N_{\mathcal{K}}(\Psi_{X})=C_{\mathcal{K}}\left(X+\mathfrak{k}_{-\frac{d}{2}}\right)$.
\end{lemma}
\begin{proof}
	Since $\Psi$ is a character, the normalizer $N_{\mathcal{K}}(\Psi_{X})$ equals
	\begin{align*}
		N_{\mathcal{K}}(\Psi_X)&=\left\{k\in\mathcal{K}\;|\; \Psi_{X}^{k}=\Psi_{X}\right\}\\
		&=\{k\in\mathcal{K}\;|\;\Psi_{X}(k^{-1}gk)=\Psi_{X}(g)\;\forall g\in\mathcal{K}_{\frac{d}{2}+}\}\\
		&=\{k\in \mathcal{K}\;|\; \psi(\mathrm{Tr}(X(k^{-1}gk-I)))=\psi(\mathrm{Tr}(X(g-I)))\;\forall g\in\mathcal{K}_{\frac{d}{2}+}\}\\
		&=\{k\in \mathcal{K}\;|\; \psi(\mathrm{Tr}((kXk^{-1}-X)(g-I)))=1\;\forall g\in\mathcal{K}_{\frac{d}{2}+}\}.
	\end{align*}
	Note that $g-I\in\mathfrak{k}_{ \frac{d}{2}+}/\mathfrak{k}_{d+}$, therefore we may write 
	$$N_{\mathcal{K}}(\Psi_X)=\{k\in \mathcal{K}\;|\; \psi(\mathrm{Tr}((kXk^{-1}-X)Y))=1\;\forall Y\in\mathfrak{k}_{ \frac{d}{2}+}\}.$$
	Hence by the definition of filtration on the dual~\eqref{filtrationofdual} we have
	\begin{align*}
		N(\Psi_X)&=\{k\in \mathcal{K}\;|\; kXk^{-1}-X\in \mathfrak{k}_{-\frac{d}{2}}\}
	\end{align*}	
	which is the centralizer of the coset $X+\mathfrak{k}_{-\frac{d}{2}}$ in $\mathcal{K}$, as was required.
\end{proof}
Thus to compute the centralizer of $\Psi_{X}$, we need to compute $C_{\mathcal{K}}(X+\mathfrak{k}_{-\frac{d}{2}})$. A direct computation yields that the centralizer $T(X)$ of $X$ in $\mathcal{K}$ is given by
\begin{equation}\label{definitionofTofX}
	T(X) = \left\{   \begin{pmatrix} a & b \\ bu^{-1}v & a \end{pmatrix}\;\middle\vert\; a,b\in \mathcal{O}_{E}, a\overline{a}+b\overline{b}u^{-1}v =1, \overline{a}b\in\sqrt{\epsilon}\mathcal{O}_{F} \right\}.
\end{equation}

\begin{proposition}\label{normalizerofpsiX}
	Let \(d \in \mathbb{Z}_{>0}\) and let \(z, u, v \in F\) such that \(\nu(z) \geq -d\) and \(\nu(v) > \nu(u) = -d\). Consider the element
	\[
	X = X(z) + \widetilde{X}(u,v) \in \mathfrak{k}_{-d}.
	\]
	Then the normalizer of \(\Psi_{X}\) in \(\mathcal{K}\) is \(T(X)\,\mathcal{K}_{\frac{d}{2}}\).
\end{proposition}
To prove the proposition, we first establish the following lemma.
\begin{lemma}\label{centralizerofXuv1}
	Let $X=\widetilde{X}(u,v)$ with $\val(u)=0$ and $\val(v)>0$. Let $s>0$.
	Let $k=\mat{a&b\\c&d}\in \mathcal{K}$. Then $k^{-1}Xk \in X+\mathfrak{k}_{s}$ if and only if $k=cg$ for some $c\in T(X)$ and $g\in \mathcal{K}_{s}$.
\end{lemma}

\begin{proof}
	Replacing $s$ with $\lceil s \rceil$ we may assume $s\in \mathbb{N}_{>0}$. Since $\mathfrak{k}_{s}$ is the intersection with $\mathfrak{g}$ of $\mathfrak{p}_E^s\mathfrak{gl}(2,\mathcal{O}_{E})$,  and conjugation by $k\in \mathcal{K}$ preserves $\mathfrak{g}$ and $\mathfrak{k}_{s}$, we can determine the elements $k\in \mathcal{K}$ that satisfy $k^{-1}Xk \in X+\mathfrak{k}_{s}$ by a system of congruences of matrix entries.  The advantage is that $kX$ is well-defined as an element of   $\mathfrak{p}_E^s\mathfrak{gl}(2,\mathcal{O}_{E})$.
	
	We note first that $T(X)\mathcal{K}_{s}\subset C_{\mathcal{K}}(X+\mathfrak{k}_{s})$.  Namely, suppose $k=cg$ with  $c\in T(X)$ and $g\in \mathcal{K}_{s}$.  Since $c\in \mathcal{K}$, it preserves $\mathfrak{k}_{s}$, so we have $c\in C_{\mathcal{K}}(X+\mathfrak{k}_{s})$. As we can write $g=I+U$ where $U$ is a matrix with all entries in $\mathfrak{p}_E^s$, it follows that $gX\equiv Xg \mod \mathfrak{p}_{E}^{s}$, where this indicates a congruence of matrix coefficients.  Therefore $gXg^{-1} \in X+\mathfrak{k}_{s}$, yielding the result.
	
	Now suppose $g\in C_{\mathcal{K}}(X+\mathfrak{k}_{s})$.  We write $g=\mat{a&b\\c&d}$ with (among other conditions) $a,b,c,d\in \mathcal{O}_{E}$ and $ad-bc\in \mathcal{O}_{E}^{\times}$.  Noting that $gXg^{-1}\equiv X$ modulo $\mathfrak{p}_E^s$ yields the equality $gX\equiv Xg$ modulo $\mathfrak{p}_E^s$. Computing the products on both sides yields the linear system
	$$
	cu\equiv bv, \quad du\equiv au \quad \text{and} \quad av\equiv dv.
	$$
	Since $u\in \mathcal{O}_{E}^{\times}$, this implies $a\equiv d \mod \mathfrak{p}_E^s$ and $c\equiv u^{-1}vb \mod \mathfrak{p}_E^s$.   Referring back to~\eqref{definitionofTofX}, we see that $k$ is congruent as a matrix modulo $\mathfrak{p}_{E}^{s}$ to an element of $T(X)$. In fact, we can argue by induction that there exists $k'\in T(X)$ such that $(k')^{-1}g\in \mathcal{K}_{s}$, whence $g\in T(X)\mathcal{K}_{s}$, as required.
	As this  inductive argument is an involved matrix calculation, we have relegated it to Appendix~\ref{section1ofappendix}.
\end{proof}

\begin{proof}[Proof of the proposition]
	By Lemma~\ref{normalizerequalscentralizerofacoset} we have $	N_{\mathcal{K}}(\Psi_{X}) = C_{\mathcal{K}}\left(X + \mathfrak{k}_{-\frac{d}{2}}\right)
	$.  Since conjugation by $\mathcal{K}$ preserves $\mathfrak{k}_{-\frac{d}{2}}$, and  $\mathcal{K}$ centralizes  $X(z)$,  it follows that
	\[
	C_{\mathcal{K}}\left(X + \mathfrak{k}_{-\frac{d}{2}}\right)
	=
	C_{\mathcal{K}}\left(\widetilde{X}(u,v) + \mathfrak{k}_{-\frac{d}{2}}\right).
	\]
	Multiplying the coset \(\widetilde{X}(u,v) + \mathfrak{k}_{-\frac{d}{2}}\) by \(\varpi^{d}\) yields $	\widetilde{X}(\varpi^{d}u,\varpi^{d}v) + \mathfrak{k}_{\frac{d}{2}}
	$, 
	and therefore
	\[
	C_{\mathcal{K}}\left(\widetilde{X}(u,v) + \mathfrak{k}_{-\frac{d}{2}}\right)
	=
	C_{\mathcal{K}}\left(\widetilde{X}(\varpi^{d}u,\varpi^{d}v) + \mathfrak{k}_{\frac{d}{2}}\right).
	\]
	Since \(\nu(u) = -d\) and \(\nu(v) > -d\), we have \(\nu(\varpi^{d}u) = 0\) and \(\nu(\varpi^{d}v) > 0\). Moreover \(\frac{d}{2} > 0\), so by Lemma~\ref{centralizerofXuv1} we obtain
	\[
	C_{\mathcal{K}}\left(\widetilde{X}(\varpi^{d}u,\varpi^{d}v) + \mathfrak{k}_{\frac{d}{2}}\right)
	=
	T(\widetilde{X}(\varpi^{d}u,\varpi^{d}v))\,\mathcal{K}_{\frac{d}{2}}
	=
	T(X)\,\mathcal{K}_{\frac{d}{2}},
	\]
	as required.
\end{proof}

First suppose that $d$ is odd. Then  $\mathcal{K}_{\frac{d}{2}+}=\mathcal{K}_{\frac{d}{2}}$. Thus in this case $\Psi_{X}$ is  a character of $\mathcal{K}_{\frac{d}{2}}$, 
and its normalizer satisfies $N(\Psi_{X}) = T(X)\,\mathcal{K}_{\frac{d}{2}}
$. Since $T(X) \subset \mathcal{K}$ is abelian, we may extend 
$\Psi_{X}\!\mid_{T(X) \cap \mathcal{K}_{\frac{d}{2}}}$ to a character of $T(X)$. 
Let $\zeta$ denote such an extension, and let $\Psi_{X,\zeta}$ be the unique extension of these characters to $T(X)\mathcal{K}_{\frac{d}{2}}$. 
Then, by part~(a) of Theorem~\ref{Cliff1},
\[
\mathrm{Ind}_{T(X)\mathcal{K}_{\frac{d}{2}}}^{\mathcal{K}} \Psi_{X,\zeta}
\]
is an irreducible representation of $\mathcal{K}$.

Next suppose that $d$ is even. Then $\mathcal{K}_{ \frac{d}{2}+}=\mathcal{K}_{\frac{d}{2}+1}$ is a proper normal subgroup of  $\mathcal{K}_{\frac{d}{2}}$, and  in this case one verifies that $\Psi_{X}$ doesn't extend to a character of $\mathcal{K}_{\frac{d}{2}}$. Therefore, we cannot apply the same simplified argument. Instead, we define the subgroup
$$
\mathcal{J}_{d}= \left\{\begin{pmatrix}1+\mathfrak{p}_{E}^{\lceil \frac{d}{2}\rceil}&\mathfrak{p}_{E}^{\lceil\frac{d}{2} \rceil}\\
	\mathfrak{p}_{E}^{\lceil\frac{d+1}{2} \rceil}& 1+\mathfrak{p}_{E}^{\lceil \frac{d}{2}\rceil} \end{pmatrix} \right\}  \cap G.
$$
Then we have proper inclusions $\mathcal{K}_{ d+\frac{1}{2}} \triangleleft \mathcal{J}_{d} \triangleleft \mathcal{K}_{\frac{d}{2}}$ of normal subgroups and  $\Psi_{X}$ extends to a character of $\mathcal{J}_{d}$. When $d$ is odd, we define $\mathcal{J}_{d} = \mathcal{K}_{\frac{d}{2}}$.

\begin{lemma}\label{normalizerJd}
	Let $d \in 2\mathbb{Z}_{>0}$. Then the normalizer of $\Psi_X$ in $T(X)\mathcal{K}_{\frac{d}{2}}$, considered as a character of $\mathcal{J}_d$, is $T(X)\mathcal{J}_d$.
\end{lemma}
\begin{proof}
	It is straightforward to verify that $T(X)\mathcal{J}_d$ centralizes $\Psi_X$. To prove the reverse inclusion, we consider a set of coset representatives
	\[
	R = \left\{ \begin{pmatrix} 1 & 0 \\ y\sqrt{\epsilon}\varpi^{\frac{d}{2}} & 1 \end{pmatrix} \;\middle|\; y \in \mathcal{O}_F^{\times} \right\}
	\]
	of $\mathcal{K}_{\frac{d}{2}} / \mathcal{J}_d$. We will show that none of these representatives normalizes $\Psi_X$. 
	As in the proof of Proposition~\ref{normalizerofpsiX} we may without loss of generality assume that $X=\widetilde{X}(u,v)\in\mathfrak{k}_{-d}$. We now show that for all $k\in R$ and  $g\in\mathcal{J}_{d}$,  $\Psi_{X}(g)\neq \Psi_{kXk^{-1}}(g)$.
	
	Let $k = \begin{pmatrix} 1 & 0 \\ y\sqrt{\epsilon}\varpi^{\frac{d}{2}} & 1 \end{pmatrix} \in R,$ and let $g = \begin{pmatrix} 1 + a\varpi^{\frac{d}{2}} & b\varpi^{\frac{d}{2}} \\ c\varpi^{\frac{d}{2}+1} & 1 - \overline{a}\varpi^{\frac{d}{2}} \end{pmatrix} \in \mathcal{J}_d.
	$ Then
	\begin{align}
		\Psi_X(g) &= \psi(\mathrm{Tr}(X(g - I))) 
		= \psi\left(  uc\sqrt{\epsilon}\varpi^{\frac{d}{2}+1} + vb\sqrt{\epsilon}\varpi^{\frac{d}{2}}  \right), \label{eq:psiX}
	\end{align}
	and
	\begin{align}
		\Psi_{kXk^{-1}}(g) &
		= \psi\left(  uc\sqrt{\epsilon}\varpi^{\frac{d}{2}+1} + vb\sqrt{\epsilon}\varpi^{\frac{d}{2}} - yua\epsilon\varpi^d - yu\overline{a}\epsilon\varpi^d \right), \label{eq:psikX}
	\end{align}
	where we have simplified by removing all terms of valuation at least $1$.
	Comparing \eqref{eq:psiX} and \eqref{eq:psikX} and writing $a=a_1+\sqrt{\epsilon}a_2$, we find that  the characters to agree if and only if
	\begin{equation}\label{eq:final}
		\psi(-2a_1yu\epsilon\varpi^d) = 1 \quad \text{for all } a_1 \in \mathcal{O}_F.
	\end{equation}
	Since \( y \in \mathcal{O}_F^\times \), \( \epsilon \in \mathcal{O}_F^\times \), and \( \nu(u) = -d \), we have \( yu\epsilon\varpi^d \in \mathcal{O}_F^\times \). Thus, as \( a_1 \) ranges over \( \mathcal{O}_F \), the expression \( -2a_1yu\epsilon\varpi^d \) ranges over all of \( \mathcal{O}_F \), and then \eqref{eq:final} contradicts the nontriviality of \( \psi \) on \( \mathcal{O}_F \). Therefore, \( \Psi_X \neq \Psi_{kXk^{-1}} \), and \( k \) does not centralize \( \Psi_X \). Hence, the normalizer of \( \Psi_X \) in \( T(X)\mathcal{K}_{\frac{d}{2}} \) is  \( T(X)\mathcal{J}_d \).
\end{proof}

We now construct the key representations of $\mathcal{K}$ that we will need in the sequel. It is a generalization of Shalika's results \cite[Theorems 4.2.1 and 4.2.5]{Sha2004} to the context of $G$.

\begin{theorem}\label{Rep of K}
	Let $X$ be as in~\eqref{element X}, and let $\zeta$ be a character of $T(X)$ which coincides with $\Psi_X$ on the intersection $T(X) \cap \mathcal{J}_{d}$.  Write $\Psi_{X, \zeta}$ for the unique character of $T(X)\mathcal{J}_{d}$ which extends $\zeta$ and $\Psi_X$. Then 
	\begin{equation}\label{repofK1}
		\mathcal{S}_d(X, \zeta) := \mathrm{Ind}_{T(X)\mathcal{J}_{d}}^\mathcal{K} \Psi_{X, \zeta}
	\end{equation}
	is an irreducible representation of $\mathcal{K}$ of depth $d$ and of degree $q^{d-1}(q^2-1)$.
\end{theorem}	
\begin{proof}
	When $d$ is odd, we already established that 
	\[
	S_{d}(X, \zeta) := \mathrm{Ind}_{T(X)\mathcal{J}_{d}}^{\mathcal{K}} \Psi_{X, \zeta}
	\]
	is an irreducible representation of $\mathcal{K}$. Now suppose that $d$ is even. Note that from Lemma~\ref{normalizerJd} the normalizer in $T(X) \mathcal{K}_{\frac{d}{2}}$ of the character $\Psi_{X}$ of $\mathcal{J}_{d}$ is $T(X)\mathcal{J}_{d}$, and $\Psi_{X, \zeta}$ is an extension of $\Psi_{X}$ to this group. Thus by~Theorem~\ref{Cliff1}, $\mathrm{Ind}_{T(X)\mathcal{J}_{d}}^{T(X)\mathcal{K}_{\frac{d}{2}}}\Psi_{X, \zeta}$ is an irreducible representation of $T(X)\mathcal{K}_{\frac{d}{2}}$ that contains $\Psi_{X}$ upon restriction to $\mathcal{K}_{\frac{d}{2}+1}$. Once again using Clifford theory for $\mathcal{K}$, and $\mathrm{Ind}_{T(X)\mathcal{J}_{d}}^{T(X)\mathcal{K}_{\frac{d}{2}}}\Psi_{X, \zeta}$ as a representation of $T(X)\mathcal{K}_{\frac{d}{2}}$ that contains $\Psi_X$ upon restriction to $\mathcal{K}_{\frac{d}{2}+1}$, we have that
	\begin{align*}
		\mathrm{Ind}_{T(X)\mathcal{K}_{\frac{d}{2}}}^{\mathcal{K}}\mathrm{Ind}_{T(X)\mathcal{J}_{d}}^{T(X)\mathcal{K}_{\frac{d}{2}}}\Psi_{X, \theta}&=\mathrm{Ind}_{T(X)\mathcal{J}_{d}}^{\mathcal{K}}\Psi_{X, \zeta}=\colon \mathcal{S}_{d}(X, \zeta)
	\end{align*}
	is an irreducible representation of  $\mathcal{K}$.
	
	As $\Psi_{X}$ has depth $d$ and $\mathcal{K}_{d+}\trianglelefteq \mathcal{K}$, we have $\mathcal{K}_{d+}\subseteq \mathrm{ker}(\mathcal{S}_{d}(X, \zeta))$, so the depth of $\mathcal{S}_d(X, \zeta)$ as a representation of $\mathcal{K}$ is $d$.

	Since $\Psi_{X, \zeta}$ is one-dimensional, the degree of  $\mathcal{S}_{d}(X, \zeta)$ is equal to the index $[\mathcal{K}\colon T(X)\mathcal{J}_{d}]$. 	We have $\mathcal{K}_{\frac{d}{2}+}\leq T(X)\mathcal{J}_{d}\leq \mathcal{K}$, and so
	$$[\mathcal{K}\colon T(X)\mathcal{J}_{d}]=\frac{[\mathcal{K}\colon \mathcal{K}_{\frac{d}{2}+}]}{[T(X)\mathcal{J}_{d}\colon \mathcal{J}_{d}][\mathcal{J}_{d}\colon \mathcal{K}_{\frac{d}{2}+}]}.$$
Since $T(X)\!\mod \mathfrak{p}_{E}\cong ZU$, $Z\cong E^{1},$ and $U\cong F$,  we have $[T(X)\colon T(X)\cap \mathcal{K}_{0+}]=q(q+1)$. The indices  $[T(X)\cap\mathcal{K}_{0+}\colon T(X)\cap \mathcal{J}_{d}], [\mathcal{J}_{d}\colon \mathcal{K}_{\frac{d}{2}+}], [\mathcal{K}_{0+}\colon \mathcal{K}_{\frac{d}{2}+}] $ are the same as the indices of corresponding $\mathcal{O}_{F}$-module in the Lie algebra. Thus we have
	\begin{align*}
		[T(X)\mathcal{J}_{d}\colon \mathcal{J}_{d}]
		&=[T(X)\colon T(X)\cap \mathcal{K}_{0+}][T(X)\cap\mathcal{K}_{0+}\colon T(X)\cap \mathcal{J}_{d}]=q(q+1)q^{2\lceil \frac{d}{2} \rceil-2},
	\end{align*}
	\begin{align*}
		[\mathcal{J}_{d}\colon \mathcal{K}_{\frac{d}{2}+}]=q^{3\left( \lceil\frac{d+1}{2} \rceil-\lceil\frac{d}{2}\rceil\right)},
	\end{align*}
	and 
	\begin{align*}	
		[\mathcal{K}\colon \mathcal{K}_{\frac{d}{2}+}]&=[\mathcal{K}\colon \mathcal{K}_{0+}][\mathcal{K}_{0+}\colon \mathcal{K}_{\frac{d}{2}+}]=|U(1,1)(\mathfrak{f})|[\mathcal{K}_{0+}\colon \mathcal{K}_{\frac{d}{2}+}]
		=q(q-1)(q+1)^{2}q^{4(\lceil \frac{d+1}{2} \rceil-1)}.
	\end{align*}
	Putting everything together we obtain
	\begin{align*}
		[\mathcal{K}\colon T(X)\mathcal{J}_{d}]=\frac{q(q-1)(q+1)^{2}q^{4(\lceil \frac{d+1}{2} \rceil-1)}}{q(q+1)q^{2\lceil \frac{d}{2} \rceil-2} q^{3\left( \lceil\frac{d+1}{2} \rceil-\lceil\frac{d}{2}\rceil\right)}}
		=(q^{2}-1)q^{\lceil \frac{d+1}{2} \rceil+\lceil\frac{d}{2} \rceil}q^{-2}=(q^{2}-1)q^{d-1}.
	\end{align*}
\end{proof}

\subsection{Representations associated to nilpotent orbits}\label{repsofKassociatedwithnilpotentorbits}

Now let us focus on a special case of particular importance: when the element $X$ is nilpotent. We will first determine a set of representatives for nilpotent orbits of $\mathcal{K}$ on $\mathfrak{g}$.

Recall that for matrix Lie algebras, an element $X$ in $\mathfrak{g}$ nilpotent if and only if $X^n=0$ for some positive integer $n$. By \cite[Theorem 3.2]{humphreys2012}, it follows that all nilpotent elements of $\mathfrak{g}$ are $G$-conjugate to a strictly upper triangular matrix. For $\delta\in F$, we define a nilpotent element $X_{\delta}\in\mathfrak{g}$ as 
\begin{equation}\label{definitionofXdelta}
	X_{\delta}:=\begin{pmatrix}0&\delta\sqrt{\epsilon}\\ 0&0 \end{pmatrix}.
\end{equation}

\begin{lemma}\label{nilpotentKorbits}
	The nilpotent $G$-orbits in $\mathfrak{g}$ are parametrized by the elements $\left\{X_{\delta} \;\middle\vert\; \delta \in \{0, 1, \varpi\} \right\}.$ Each $G$-orbit further decomposes as a disjoint union of infinitely many $\mathcal{K}$-orbits,
	parametrized by $\left\{X_{0},\, X_{\varpi^{-d}} \mid d \in \mathbb{Z}\right\}.$
\end{lemma}
\begin{proof}
	Let $X_{\delta_1}, X_{\delta_2} \in \mathfrak{g}$ with $\delta_1, \delta_2 \in F$, and suppose there exists
	\[
	g = \begin{pmatrix}a & b \\ c & d\end{pmatrix} \in G \quad \text{such that} \quad gX_{\delta_1}g^{-1} = X_{\delta_2}.
	\]
	A direct calculation shows that if $g$ satisfies this relation, then necessarily $c = 0$, $d = \overline{a}^{-1}$, and $a\overline{a}\,\delta_1 = \delta_2.$
	
Since $a \in E^{\times}$, we deduce that $\delta_2 \in \delta_1 \cdot N_{E/F}(E^{\times})$. In particular, if $\delta_1 = 0$, then $\delta_2 = 0$, so $X_{0}$ lies in its own orbit. For $\delta_1 \neq 0$, the $G$-orbits of $X_{\delta}$ are thus parametrized by the cosets of $F^{\times} / N_{E/F}(E^{\times})$. By local class field theory~(\cite[Chapter~VI, Theorem~2]{Cas1967}), this quotient has index~\(2\). Since \( \varpi \notin N_{E/F}(E^{\times}) \), the set \( \{1, \varpi\} \) forms a complete set of coset representatives for \( F^{\times} / N_{E/F}(E^{\times}) \), proving the first claim.
	
For the $\mathcal{K}$-orbits, we now require $g \in \mathcal{K}$. Then the above computation forces $a \in \mathcal{O}_{E}^{\times}$, so $a\overline{a} \in \mathcal{O}_{F}^{\times}$, and hence $a\overline{a}\,\delta_1 = \delta_2$ implies $\nu(\delta_1) = \nu(\delta_2)$. Conversely, if $\nu(\delta_1)=\nu(\delta_2)$, then $\delta_2/\delta_1 \in \mathcal{O}_{F}^{\times}$, and since $E/F$ is unramified, the norm map $N_{E/F} : \mathcal{O}_{E}^{\times} \to \mathcal{O}_{F}^{\times}$ is surjective, so we may choose $a \in \mathcal{O}_{E}^{\times}$ such that $a\overline{a} = \delta_2/\delta_1$, which gives $gX_{\delta_1}g^{-1} = X_{\delta_2}$ with $g \in \mathcal{K}$. This establishes the stated parametrization of $\mathcal{K}$-orbits.
\end{proof}

For $\delta \in\{0, 1, \varpi\}$, let $\mathcal{N}_{\delta}$ denote the $G$-orbit of $X_{\delta}$. By the proof of Lemma~\ref{nilpotentKorbits} each $G$-orbit $\mathcal{N}_{\delta}$ further decomposes as follows
$$\mathcal{N}_{0}=\mathcal{K}\cdot X_{0},\hspace{3em}\mathcal{N}_{1}=\bigsqcup_{d\in 2\mathbb{Z}}\mathcal{K}\cdot X_{\varpi^{-d}},\hspace{3em}\mathcal{N}_{\varpi}=\bigsqcup_{d\in 2\mathbb{Z}+1}\mathcal{K}\cdot X_{\varpi^{-d}}.$$

Let $d\in\mathbb{Z}_{>0}$, and $a\in\mathcal{O}_{F}^{\times}$. Then $X_{a\varpi^{-d}}\in\mathfrak{k}_{-d}$. We now apply the technique \eqref{characterdefinition} of the preceding section to $X_{a\varpi^{-d}}$ to obtain a character $\Psi_{X_{a\varpi^{-d}}}$ of $\mathcal{J}_{d}$ given by $\Psi_{X_{a\varpi^{-d}}}(g)=\psi(\mathrm{Tr}(X_{a\varpi^{-d}}(g-I)))$. By direct computation it follows that the centralizer $T(X_{a\varpi^{-d}})$ of $X_{a\varpi^{-d}}$ in $\mathcal{K}$ is equal to $(ZU)\cap \mathcal{K}$ where
$$U=\left\{\begin{pmatrix} 1& \sqrt{\epsilon} b\\0&1\end{pmatrix}\;\middle\vert\;b\in F\right\}.$$
Since $Z\subset\mathcal{K}$, $(ZU)\cap\mathcal{K}=Z(U\cap\mathcal{K})$. For convenience, set $\mathcal{U}:=U\cap\mathcal{K}$. By Proposition~\ref{normalizerofpsiX} the normalizer  of $\Psi_{X_{a\varpi^{-d}}}$ in $\mathcal{K}$ is given by $Z\mathcal{U}\mathcal{K}_{\frac{d}{2}}$. Observe that $\Psi_{X_{a\varpi^{-d}}}$ acts trivially on $Z\mathcal{U}\cap \mathcal{J}_{d}$.

We are interested only in the special case where the character $\zeta$ of $T(X)=Z\mathcal{U}$ is also trivial on $\mathcal{U}$, so that it is entirely determined by its value on $Z$. In this case, let $\theta$ be a character of $Z$ such that $\theta\mid _{Z\cap\mathcal{J}_{d}}=\mathbbm{1}$. We  extend $\theta$ trivially to $\mathcal{U}$ and  write $\Psi_{X_{a\varpi^{-d}}, \theta}$ for the unique character of $Z\mathcal{U}\mathcal{J}_{d}$ that extends $\theta$ and $\Psi_{X_{a\varpi^{-d}}}$. Then from Theorem~\ref{Rep of K} we have that
\begin{equation}\label{nilpotentrep1}
	\mathcal{S}_d(X_{a\varpi^{-d}}, \theta) := \mathrm{Ind}_{Z\mathcal{U}\mathcal{J}_{d}}^\mathcal{K} \Psi_{X_{a\varpi^{-d}}, \theta}
\end{equation}
is an irreducible representation of $\mathcal{K}$ of depth $d$ and of degree $q^{d-1}(q^2-1)$.
\begin{lemma}\label{remark3.3.4}
	Let $a \in \mathcal{O}_{F}^{\times}$. Then $\mathcal{S}_{d}(X_{\varpi^{-d}}, \theta)\cong \mathcal{S}_{d}(X_{a\varpi^{-d}}, \theta).$
\end{lemma}
\begin{proof}
	By Theorem~\ref{Cliff1}, conjugation by any $k \in \mathcal{K}$ sends $\mathcal{S}_{d}(X_{\varpi^{-d}}, \theta)$ to $\mathcal{S}_{d}(\mathrm{Ad}(k)X_{\varpi^{-d}}, \theta^{k}).$
	Since $\theta$ is a character of the center $Z$, we have $\theta^{k} = \theta$. Moreover, $X_{\varpi^{-d}}$ and $X_{a\varpi^{-d}}$ lie in the same $\mathcal{K}$-orbit for all $a \in \mathcal{O}_{F}^{\times}$. The result follows.
	\end{proof}

We conclude this section with a lemma that relates the characters $\Psi_{X}$ and $\Psi_{X_{u}}$ in a special but very informative case.

\begin{lemma}\label{d>2rcong}
	Let $X = X(z) + \widetilde{X}(u,v) \in \mathfrak{k}_{-d}$ be such that $\nu(z), \nu(v) > -\lceil \tfrac{d}{2} \rceil$ and $\nu(u) = -d$. Then
	\[
	\Psi_{X} = \Psi_{X_{u}}
	\qquad \text{and} \qquad
	T(X)\mathcal{J}_{d} = T(X_{u})\mathcal{J}_{d} = Z\mathcal{U}\mathcal{J}_{d}.
	\]
\end{lemma}

\begin{proof}
	Since $X \equiv X_{u}$ entry-wise modulo $\mathfrak{p}_{E}^{-\lceil d/2 \rceil}$, it follows that $\Psi_{X} = \Psi_{X_{u}}$ as characters of $\mathcal{J}_{d}$. Hence, they have the same normalizer in $T(X)\mathcal{K}_{\frac{d}{2}}$. By Proposition~\ref{normalizerofpsiX} and Lemma~\ref{normalizerJd}, we obtain
	\[
	N_{T(X)\mathcal{K}_{\frac{d}{2}}}(\Psi_{X}) = T(X)\mathcal{J}_{d}
	\qquad \text{and} \qquad
	N_{T(X)\mathcal{K}_{\frac{d}{2}}}(\Psi_{X_{u}}) = T(X_{u})\mathcal{J}_{d} = Z\mathcal{U}\mathcal{J}_{d},
	\]
	which proves the lemma.
\end{proof}

\section{An explicit description of $\mathcal{W}_{d,\chi}$}\label{an explicit decomposition}
In \S\ref{a canonical decomposition} we obtained a canonical decomposition of $\mathrm{Ind}_{\mathcal{B}}^{\mathcal{K}}\chi$.
For $d \geq r+1$, however, the representations $\mathcal W_{d,\chi}$ appear only as quotients of certain induced representations. In this section, we show that  each of these representations are isomorphic to one of the explicit representations that we constructed in Theorem~\ref{Rep of K}.

We now state the main decomposition theorem that we prove in this section.

\begin{theorem}\label{prop2}
	Let $\chi$ be a character of $T$ of minimal depth $r\in \mathbb{Z}_{\geq 0}$. Then for each $d\geq r+1$, there exists $Y_{\chi}\in\mathfrak{k}_{-d}$, and $\zeta_{\chi}$ a character of the centralizer of $Y_{\chi}$ in $\mathcal{K}$, satisfying $	\zeta_{\chi}|_{T(Y_{\chi})\cap \mathcal{J}_{d}}=\Psi_{Y_{\chi}}|_{T(Y_{\chi})\cap \mathcal{J}_{d}},$ such that
	$$\mathcal{W}_{d, \chi}\cong \mathcal{S}_{d}(Y_{\chi}, \zeta_{\chi}).$$
\end{theorem}	
For all $d\geq r+1$,  our strategy is to first identify elements $Y_{\chi}$ and $\zeta_{\chi}$ such that $\mathcal{S}_{d}(Y_{\chi}, \zeta_{\chi})$ is a well-defined irreducible representation in the sense of Theorem~\ref{Rep of K} and the homomorphism space
\begin{equation}\label{principalseriesYchiidentification}\mathrm{Hom}_{\mathcal{K}}(V_{\chi}^{\mathcal{K}_{d+1}}, \mathcal{S}_{d}(Y_{\chi}, \zeta_{\chi}))\end{equation}
is not equal to zero, \emph{i.e.}, $\mathcal{S}_{d}(Y_{\chi}, \zeta_{\chi})$ is a subrepresentation of $V_{\chi}^{\mathcal{K}_{d+1}}$.

We now work toward identifying \( Y_{\chi} \) and \( \zeta_{\chi} \). Determining these objects is a challenging task, and a major portion of this section is devoted to that goal. Applying Frobenius reciprocity and Mackey theory to~\eqref{principalseriesYchiidentification} we obtain
\begin{align*}
	\mathrm{Hom}_{\mathcal{K}}(\mathrm{Ind}_{\mathcal{B}\mathcal{K}_{d+1}}^{\mathcal{K}}\chi, \mathrm{Ind}_{T(Y_{\chi})\mathcal{J}_{d}}^{\mathcal{K}}\Psi_{Y_{\chi}, \zeta_{\chi}})\cong&\bigoplus_{\alpha\in \mathcal{B}\mathcal{K}_{d+1}\backslash \mathcal{K}/ T(Y_{\chi})\mathcal{J}_{d}}\mathrm{Hom}_{\mathcal{B}\mathcal{K}_{d+1}\cap (T(Y_{\chi})\mathcal{J}_{d})^{\alpha} }(\chi, \Psi_{Y_{\chi}, \zeta_{\chi}}^{\alpha}).
\end{align*}	
Let $hk\in \mathcal{B}\mathcal{K}_{d+1}\cap (T(Y_{\chi})\mathcal{J}_{d})^{\alpha}$ where $h\in \mathcal{B}$ and $k\in\mathcal{K}_{d+1}$, and let $\alpha\in \mathcal{B}\mathcal{K}_{d+1}\backslash \mathcal{K}/ T(Y_{\chi})\mathcal{J}_{d}$. Note that $\chi$ is trivial on $\mathcal{K}_{d+1}$ by definition; therefore $\chi(hk)=\chi(h)$. Since $\mathcal{K}_{d+1}\trianglelefteq\mathcal{K}$, and $\Psi_{Y_{\chi}, \zeta_{\chi}}$ has depth $d$, we have $\alpha^{-1}k\alpha\in\mathcal{K}_{d+1}$ and $\Psi_{Y_{\chi}, \zeta_{\chi}}^{\alpha}(k)=\Psi_{Y_{\chi}, \zeta_{\chi}}^{\alpha}(\alpha^{-1}k\alpha)=1$. Thus it suffices to show that for some choice of coset representative $\alpha$, and  $Y_{\chi}, \zeta_{\chi}$,
$$\chi(h)=\Psi_{Y_{\chi}, \zeta_{\chi}}^{\alpha}(h),\hspace{2em}\mathrm{for\;all}\;h\in\mathcal{B}\cap(T(Y_{\chi})\mathcal{J}_{d})^{\alpha}.$$
Note that if $h\in\mathcal{B}\cap(T(Y_{\chi})\mathcal{J}_{d})^{\alpha}$, then $h=\alpha t\alpha^{-1}\alpha j \alpha^{-1}$ for some $t\in T(Y_{\chi})$, and $j\in\mathcal{J}_{d}$. We now evaluate $\Psi_{Y_{\chi}, \zeta_{\chi}}^{\alpha}(h)=\Psi_{Y_{\chi}, \zeta_{\chi}}(tj)=\zeta_{\chi}(t)\Psi_{Y_{\chi}}(j)$. If there exists a choice of $\alpha$ such that $\alpha t \alpha^{-1}$ is upper triangular, then $h$ being upper triangular forces $\alpha j \alpha^{-1}$ to be upper triangular. For such a choice of $\alpha$ we would have $\chi(h)=\chi(\alpha t \alpha^{-1})\chi(\alpha j \alpha^{-1})$. Thus it suffices to find $Y_{\chi}, \zeta_{\chi}$  such that $\mathcal{S}_{d}(Y_{\chi}, \zeta_{\chi})$ is a well-defined and
\begin{equation}\label{mainequationtocheck}
	\zeta_{\chi}(t)=\chi(\alpha t\alpha^{-1}),\;\;\;\mathrm{and}\;\;\; \Psi_{Y_{\chi}}(j)=\chi(\alpha j \alpha^{-1}),\end{equation}
for some choice of coset representative $\alpha$ that makes $t$ upper triangular whenever  $h=\alpha t \alpha^{-1}\alpha j \alpha^{-1}\in \mathcal{B}\cap(T(Y_{\chi})\mathcal{J}_{d})^{\alpha}$, where $j\in \mathcal{J}_{d}$.

We first assume that $\chi$ has depth $r>0$.   Then $\chi\mid_{T_{\frac{r}{2}+}/ T_{r+}}$ is realized by an element $\Gamma=\begin{pmatrix}
	x&0\\0&-\overline{x}
\end{pmatrix}\in\mathfrak{t}_{-r}/\mathfrak{t}_{\frac{-r}{2}}$ with $x\in\mathfrak{p}_{E}^{-r}$, \emph{i.e.,} for all $t\in T_{\frac{r}{2}+}/ T_{r+} $ we have
\begin{equation}\label{Shalika1}
	\chi(t)=\psi(\mathrm{Tr}(\Gamma e^{-1}(t))).
\end{equation}
where  $e:\mathfrak{t}_{\frac{r}{2}+}/\mathfrak{t}_{r+}\rightarrow T_{\frac{r}{2}+}/ T_{r+}$ denotes the Moy--Prasad isomorphism.
	In particular, for $t\in T_{\frac{r}{2}+}/ T_{r+} $ we have
	\begin{equation}\label{eq:chi on T}
		\chi(t)=\psi(\mathrm{Tr}(\Gamma(t-I))).
	\end{equation}
Let  $d\geq r+1$. Recall that in ~\S\ref{chapter3}, $\zeta_{\chi}$ was the character of the centralizer of some element $Y_{\chi}\in \mathfrak{k}_{-d}$ having the given form. If we can find $Y_{\chi}$ such that it is conjugate to $\Gamma$ by an element of $G$, then we can easily relate the characters $\chi$ and $\zeta_{\chi}$. Since $\Gamma$ is a diagonal matrix we just need to equate the eigenvalues of $\begin{pmatrix}z\sqrt{\epsilon}&u\sqrt{\epsilon}\\v\sqrt{\epsilon}&z\sqrt{\epsilon}\end{pmatrix}$ with that of $\Gamma$. Therefore we should have
	$$z\sqrt{\epsilon}+\sqrt{uv\epsilon}=x_1+x_2\sqrt{\epsilon},\;\; z\sqrt{\epsilon}-\sqrt{uv\epsilon}=-x_1+x_2\sqrt{\epsilon}.$$
	Solving this yields, $z=x_2$, and $uv\epsilon=x_1^{2}$. We also want $\nu(v)>\nu(u)=-d$, and $ \nu(z)\geq-d$. Therefore, one possible choice is  to take $u=\epsilon^{-1}\varpi^{-d}$, and $v=x_{1}^{2}\varpi^{d}$.
	Since $d>r$, and $x_{i}\in \mathfrak{p}_{E}^{-r}$, the above choices of $u$ and $v$ satisfy the valuation conditions. Thus we set
	$$Y_{\chi}:=\begin{pmatrix}
		x_2\sqrt{\epsilon}& \epsilon^{-1}\varpi^{-d}\sqrt{\epsilon}\\ x_1^{2}\varpi^{d}\sqrt{\epsilon}&x_2\sqrt{\epsilon}  
	\end{pmatrix}\in\mathfrak{k}_{-d}.$$
	Let $\gamma=x_1\varpi^{d}\sqrt{\epsilon}\in\mathfrak{p}_{E}^{d-r}$. Then a matrix $g_{d}$ in $G$ that conjugates $Y_{\chi}$ to $\Gamma$ is given by $$g_{d}:=\begin{pmatrix}
		1&-\frac{1}{2}\gamma^{-1}\\\gamma&\frac{1}{2} 
	\end{pmatrix}\in G,$$ and we have $Y_{\chi}=\Gamma^{g_{d}}$. Since the centralizer of $\Gamma$ in $G$ is $T$, the centralizer $T(Y_{\chi})$ of $Y_{\chi}$ in $\mathcal{K}$ equals to the centralizer of $\Gamma^{g_{d}}$ in $\mathcal{K}$ which is $T^{g_{d}}\cap\mathcal{K}$.  As $\chi$ is a character of $T$, $\zeta_{\chi}:=\chi^{g_d}$ gives a character of $T(Y_{\chi})$. 
	A direct computation gives that
	$$T(Y_{\chi})=\left\{\begin{pmatrix} a&b\\ b\gamma^{2}&a\end{pmatrix}\mid \overline{a}a+\overline{b}b\gamma^2=1, \overline{a}b\in \sqrt{\epsilon}F\right\}\cap \mathcal{K}.$$
	We choose the coset representative $\alpha=\begin{pmatrix} 1&0\\ -\gamma&1\end{pmatrix}$. Then for any $t=\begin{pmatrix} a& b\\ b\gamma^2&a\end{pmatrix}\in T(Y_{\chi})$, we have
	$$\alpha t \alpha^{-1}=\begin{pmatrix}a+b\gamma & b\\ 0& a-b\gamma\end{pmatrix}$$
	and $$\chi(\alpha t\alpha^{-1})=\chiE(a+b\gamma)=\chi^{g_{d}}(t)=\zeta_{\chi}(t),$$
	verifying the first relation in~\eqref{mainequationtocheck}.
	
	Now let $j=\begin{pmatrix}
		1+c\varpi^{\lceil\frac{d}{2}\rceil}& y\varpi^{\lceil\frac{d}{2}\rceil}\\ z\varpi^{\lceil\frac{d+1}{2}\rceil}&1-\overline{c}\varpi^{\lceil\frac{d}{2}\rceil}
	\end{pmatrix}  \in\mathcal{J}_{d}$,  where $ c, y, z\in\mathcal{O}_{E}$ such that  $h=\alpha t\alpha^{-1}\alpha j \alpha^{-1} \in\mathcal{B}\cap(T(Y_{\chi})\mathcal{J}_{d})^{\alpha}$ where $t\in T(Y_{\chi})$. Then
	$$\alpha j \alpha^{-1}=\begin{pmatrix}1+c\varpi^{\lceil\frac{d}{2}\rceil}+y\gamma\varpi^{\lceil\frac{d}{2}\rceil}&y\varpi^{\lceil\frac{d}{2}\rceil}\\-c\gamma\varpi^{\lceil\frac{d}{2} \rceil}+z\varpi^{\lceil\frac{d+1}{2} \rceil}-y\gamma^2\varpi^{\lceil\frac{d}{2} \rceil}-\overline{c}\gamma\varpi^{\lceil\frac{d}{2}\rceil}&-y\gamma\varpi^{\lceil\frac{d}{2} \rceil}+1-\overline{c}\varpi^{\lceil\frac{d}{2}\rceil}
	\end{pmatrix}.$$
	Since $h$ and $\alpha t\alpha^{-1}$ are upper triangular,  $\alpha j\alpha^{-1}$ must be upper triangular as well. Therefore, the $(2, 1)$ matrix entry of $\alpha j \alpha^{-1}$ vanishes; that is,
	\begin{equation}\label{6.6}
		-c\gamma\varpi^{\lceil\frac{d}{2} \rceil}+z\varpi^{\lceil\frac{d+1}{2} \rceil}-y\gamma^2\varpi^{\lceil\frac{d}{2} \rceil}-\overline{c}\gamma\varpi^{\lceil\frac{d}{2}\rceil}=0.
	\end{equation}
	Multiplying~\eqref{6.6} by $\gamma^{-1}x_1$, and grouping the terms with $c$ together we obtain
	\begin{equation}\label{equation:simplication11}
		(c+\overline{c})x_1\varpi^{\lceil\frac{d}{2} \rceil}+x_1y\gamma\varpi^{\lceil\frac{d}{2} \rceil}-\epsilon^{-1}z\sqrt{\epsilon}\varpi^{\lceil\frac{d+1}{2} \rceil-d}=0.
	\end{equation}
	Since $1+c\varpi^{\lceil \frac{d}{2}\rceil}+y\gamma \varpi^{\lceil \frac{d}{2} \rceil}\in\mathfrak{p}_{E}^{\frac{d}{2}}\subseteq \mathfrak{p}_{E}^{\frac{r}{2}+}$, applying~\eqref{eq:chi on T} we obtain
	\begin{equation}
		\chi(\alpha j \alpha^{-1})=\psi(\mathrm{Tr}(\Gamma(\alpha j\alpha^{-1}-I)))=\Psi_{\Gamma^{\alpha^{-1}}}(j).
	\end{equation}	
	From~\eqref{equation:simplication11} it follows that
	\begin{align*}
		\psi(\mathrm{Tr}(\Gamma^{\alpha^{-1}}-Y_{\chi})(j-I))&=\psi((c+\overline{c})x_1\varpi^{\lceil\frac{d}{2} \rceil}+x_1y\gamma\varpi^{\lceil\frac{d}{2} \rceil}-\epsilon^{-1}z\sqrt{\epsilon}\varpi^{\lceil\frac{d+1}{2} \rceil-d})=\psi(0)=1.
	\end{align*}	
	Hence $\chi(\alpha j \alpha^{-1})=\Psi_{Y_{\chi}}(j)$, thus verifying the second relation in~\eqref{mainequationtocheck}.

	It remains to show that $\Psi_{Y_{\chi}}$, and $\zeta_{\chi}$ agree on the intersection $T(Y_{\chi})\cap \mathcal{J}_{d}$.  	Indeed, let $h:=\begin{pmatrix} a&b\\ b\gamma^{2}&a\end{pmatrix}\in T(Y_{\chi})\cap\mathcal{J}_{d}$. Then 
	$g_{d}^{-1}hg_{d}=\begin{pmatrix}a+b\gamma &0\\ 0&a-b\gamma\end{pmatrix}\in T_{\lceil \frac{d}{2} \rceil}\subset T_{\frac{r}{2}+}.$ Therefore, applying~\eqref{eq:chi on T} we obtain
	$$\zeta_{\chi}(h)=\chi^{g_{d}}(h)=\chi(g_{d}^{-1}hg_{d})=\psi(\mathrm{Tr}(\Gamma(g_{d}^{-1}hg_{d}-I)))=\psi(\mathrm{Tr}(g_{d}\Gamma g_{d}^{-1}(h-I)))=\Psi_{Y_{\chi}}(h).$$
Let $\Psi_{Y_{\chi}, \zeta_{\chi}}$ denote the unique extension of $\zeta_{\chi}$ and $\Psi_{Y_{\chi}}$ to a character of $T(Y_{\chi})\mathcal{J}_{d}$. Then from Theorem~\ref{Rep of K} it follows that 
	\begin{equation}\label{prinshalika1}
		\mathcal{S}_{d}(Y_{\chi}, \zeta_{\chi})=\mathrm{Ind}_{T(Y_{\chi})\mathcal{J}_{d}}^{\mathcal{K}}\Psi_{Y_{\chi}, \zeta_{\chi}}
	\end{equation} 	
is an irreducible representation of $\mathcal{K}$ of depth $d$ and degree $(q^{2}-1)q^{d-1}$. Thus, given a character $\chi$ of minimal depth $r\in\mathbb{Z}_{>0}$, for each $d\geq r+1$, we have constructed an irreducible  representation $\mathcal{S}_{d}(Y_{\chi}, \zeta_{\chi})$
	which is a subrepresentation of $V_{\chi}^{\mathcal{K}_{d+1}}$.
	
	We now turn our attention to the case when $\chi$ has depth-zero. In this case, there are no elements of $\mathfrak{t}$ realizing $\chi$. Instead, for all $d\geq 1$, we take $$Y_{\chi}=X_{\varpi^{-d}}=\begin{pmatrix}0& \varpi^{-d}\sqrt{\epsilon}\\0&0\end{pmatrix},$$	and set $\zeta_{\chi}=\theta$, where $\theta=\chi\mid_{Z}$ is the central character of $\pi_{\chi}$. In this case, the centralizer of $X_{\varpi^{-d}}$ in $\mathcal{K}$ equals $Z\mathcal{U}\mathcal{J}_{d}$, and $\zeta_{\chi}\mid_{Z\mathcal{U}\cap\mathcal{J}_{d}}=\Psi_{Y_{\chi}}\mid_{Z\mathcal{U}\cap\mathcal{J}_{d}}=\mathbbm{1} $; thus applying ~\eqref{nilpotentrep1} we obtain that 
	\begin{equation}\label{prinshalika2}\mathcal{S}_{d}(X_{\varpi^{-d}}, \theta)=\mathrm{Ind}_{Z\mathcal{U}\mathcal{J}_{d}}^{\mathcal{K}}\Psi_{X_{\varpi^{-d}}, \theta}\end{equation} is an irreducible representation of $\mathcal{K}$ of depth $d$ and degree $(q^{2}-1)q^{d-1}$. Moreover, if $h\in \mathcal{B}\cap Z\mathcal{U}\mathcal{J}_{d} $, then $h=tj$, where $t\in Z\mathcal{U} $ and $j$ is an upper triangular matrix in $\mathcal{J}_{d}$ on which $\chi$ is therefore trivial. Hence we have  
		$$\chi(h)=\chi(t)\chi(j)=\chi(t)=\zeta_{\chi}(t)=\Psi_{X_{\varpi^{-d}}, \theta}(h).$$
Hence, $S_{d}(X_{\varpi^{d}}, \theta)$ is a subrepresentation of $V_{\chi}^{\mathcal{K}_{d}}$.

		\begin{proof}[Proof of the main theorem]
			To show the required isomorphism of $\mathcal{K}$-representations, we begin by recalling that we have already established the decomposition  
			\[
			V_{\chi}^{\mathcal{K}_{d+1}} = V_{\chi}^{\mathcal{K}_{d}} \oplus \mathcal{W}_{d, \chi}.
			\]  
			Since any irreducible component of \( V_{\chi}^{\mathcal{K}_{d}} \) must have depth at most \( d-1 \), while \( S_{d}(Y_{\chi}, \zeta_{\chi}) \) has depth exactly \( d \), it follows that \( S_{d}(Y_{\chi}, \zeta_{\chi}) \) cannot be isomorphic to a subrepresentation of \( V_{\chi}^{\mathcal{K}_{d}} \) and must therefore be isomorphic to a subrepresentation of \( \mathcal{W}_{d, \chi} \). But \( \mathcal{W}_{d, \chi} \) is irreducible and has the same dimension as \( S_{d}(Y_{\chi}, \zeta_{\chi}) \), hence we conclude that  
			\[
			\mathcal{W}_{d, \chi} \cong S_{d}(Y_{\chi}, \zeta_{\chi}).
			\]
		\end{proof}

		The following is now an immediate corollary of Lemma~\ref{principalseries2}, and Theorems~\ref{thm2} and~\ref{prop2}.
		
		\begin{corollary}\label{thm3}
			Let $\chi$ be a character of \( T \) of minimal depth \( r \in \mathbb{Z}_{\geq 0} \). Then the restriction of the principal series representation \( \pi_{\chi} \) to $\mathcal{K}$ decomposes as direct sum of irreducible representations of $\mathcal{K}$ as follows
			\begin{equation*}
				\mathrm{Res}_{\mathcal{K}} \pi_{\chi} \cong
				\begin{cases}
					\mathrm{Ind}_{\mathcal{B}\mathcal{K}_{r+1}}^{\mathcal{K}} \chi \oplus \displaystyle\bigoplus_{d \geq r+1} \mathcal{S}_{d}(Y_{\chi}, \zeta_{\chi}) & \text{if } r > 0, \\
					\mathrm{Ind}_{\mathcal{B}\mathcal{K}_{1}}^{\mathcal{K}} \chi \oplus \displaystyle\bigoplus_{d \geq 1} \mathcal{S}_{d}(X_{\varpi^{-d}}, \theta) & \text{if } r = 0 \text{ and } \chi \neq \mathbbm{1}, \\
					\mathbbm{1}_{q} \oplus \mathrm{St}_{q} \oplus \displaystyle\bigoplus_{d \geq 1} \mathcal{S}_{d}(X_{\varpi^{-d}}, \mathbbm{1}) & \text{if } r = 0 \text{ and } \chi = \mathbbm{1}.
				\end{cases}
			\end{equation*}
			Here, $Y_{\chi} = \Gamma^{g_d}, \quad \text{and} \quad \zeta_{\chi} = \chi^{g_d},$ where \( \Gamma = \begin{pmatrix} x & 0 \\ 0 & -\overline{x} \end{pmatrix} \in \mathfrak{t}_{-r}/\mathfrak{t}_{-r/2} \) is an element that realizes \( \chi \) on \( T_{r/2+} / T_{r+} \), and \[g_d = \begin{pmatrix}
				1 & -\frac{1}{2} \gamma^{-1} \\
				\gamma & \frac{1}{2}
			\end{pmatrix}\in G,\; \text{with } \gamma = \frac{(x + \overline{x}) \varpi^{d} \sqrt{\epsilon}}{2}.\]
		\end{corollary}
		

\section{Applications}\label{restrictiontoK2rprincipalseries}		
		
Let $\chi$ be a character of $T$ of minimal depth $r \in \mathbb{Z}_{\geq 0}$, and let $\pi_{\chi}$ be the associated principal series representation. 
		By Lemma~\ref{centralcharacterofprincipalserieshasdepth-zero}, we may assume that the central character $\theta$ of $\pi_{\chi}$ is either $\mathbbm{1}$ or $\delta$ at the expense of twisting our branching rule by a character of $G$, where $\delta$ denotes the nontrivial quadratic character of $E^{1}$.

In this section, we study the relationship between the higher-depth components appearing in the decomposition of $\pi_{\chi}$ and those appearing in the decomposition of depth-zero principal representations with the same central character as $\pi_{\chi}$. We also examine the restriction of $\pi_{\chi}$ to the subgroup $\mathcal{K}_{2r+} \subseteq \mathcal{K}$ and show that its decomposition is entirely determined by the representations of the form~\eqref{nilpotentrep1}, constructed from the nilpotent orbits in the Lie algebra of $G$ together with the character $\theta$ of the center of $G$. This establishes a new case of a recent conjecture in the literature~\cite{ZanMon2025, Mon2024, henniart2024representationssl2f, guy2024representationsglndnearidentity}.
		
	\subsection{Connection with depth-zero representations}
In this section we prove that  the higher-depth components appearing in the decomposition upon restriction to $\mathcal{K}$ of a depth-zero principal series representation are identical to those appearing in the corresponding decomposition of any principal series representation, provided the two representations share the same central character.

Note that if $\chi$ is a character of minimal depth zero with central character $\theta$, then by Corollary~\ref{thm3}, for every $d \geq 1$, the components occurring in the decomposition of $\mathrm{Res}_{\mathcal{K}}\pi_{\chi}$ are of the form $\mathcal{S}_{d}(X_{\varpi^{-d}}, \theta)$.

\begin{theorem}\label{keyidentificationp}
Let $\chi$ be a character of $T$ of minimal depth $r>0$, and let $\mathcal{S}_{d}(Y_{\chi}, \zeta_{\chi})$ be as in~Corollary~\ref{thm3}. Then, for $d > 2r$, we have an isomorphism 
\[\mathcal{S}_{d}(Y_{\chi}, \zeta_{\chi}) \cong \mathcal{S}_{d}(X_{\varpi^{-d}}, \theta).\]
\end{theorem}
\begin{proof}
	We have $Y_{\chi}=\Gamma^{g_d}=\begin{pmatrix}
		x_2\sqrt{\epsilon}& \epsilon^{-1}\varpi^{-d}\sqrt{\epsilon}\\ x_1^{2}\varpi^{d}\sqrt{\epsilon}&x_2\sqrt{\epsilon}  
	\end{pmatrix}\in\mathfrak{k}_{-d}$. Consider the nilpotent element $X_{\epsilon^{-1}\varpi^{-d}}$. Since $d>2r$,  we have $-r>-\frac{d}{2}$, and therefore
	$$\nu(x_{2})\geq -r> -\frac{d}{2}\geq-\left\lceil\frac{d}{2} \right\rceil\hspace{2em}\text{and}\hspace{2em}\nu(x_1^{2}\varpi^{d}\sqrt{\epsilon})\geq d-2r>0>-\left\lceil \frac{d}{2}\right\rceil. $$
	Thus  by Lemma~\ref{d>2rcong}, we have  $\Psi_{Y_{\chi}}=\Psi_{X_{\epsilon^{-1}\varpi^{d}}}$ on $\mathcal{J}_{d}$ and  $T(Y_{\chi})\mathcal{J}_{d}=T(X_{\epsilon^{-1}\varpi^{-d}})=Z\mathcal{U}\mathcal{J}_{d}$.
	Since $\Psi_{X_{\epsilon^{-1}\varpi^{-d}}}$ is trivial on the intersection $Z\mathcal{U}\cap \mathcal{J}_{d}$ and we assumed that $\theta$ has depth-zero, $\theta$ and $\Psi_{X_{\epsilon^{-1}\varpi^{-d}}}$  agree on $Z\mathcal{U}\cap \mathcal{J}_{d}$. Therefore, applying~\eqref{nilpotentrep1}, we obtain that $\mathcal{S}_{d}(X_{\epsilon^{-1}\varpi^{-d}}, \theta)$ is an irreducible representation of $\mathcal{K}$ of depth $d$ and degree $q^{d-1}(q^{2}-1)$.
	
	We claim that
	$\mathcal{S}_{d}(Y_{\chi}, \zeta_{\chi})=\mathcal{S}_{d}(X_{\epsilon^{-1}\varpi^{-d}}, \theta).$ It suffices to show that 
	$$\Psi_{Y_{\chi}, \zeta_{\chi}}=\Psi_{X_{\epsilon^{-1}}, \theta}\quad \text{on}\quad Z\mathcal{U}\mathcal{J}_{d}.$$
	Since $\Psi_{Y_{\chi}}=\Psi_{X_{\epsilon^{-1}\varpi^{-d}}}$ on $\mathcal{J}_{d}$, we only need to show that $\zeta \mid_{T(Y_{\chi})}=\theta\mid_{Z}$. Indeed, let $t:=\begin{psmallmatrix}a&b\\ b\gamma^{2}&a\end{psmallmatrix}\in T(Y_{\chi})$. Then $\zeta_{\chi}(t)=\chi^{g_{d}}(t)=\chi(t^{g_{d}^{-1}})$.	Note that $\gamma \in \mathfrak{p}_{E}^{d-r}$, as $d>2r$, we have that $\gamma \in \mathfrak{p}_{E}^{r+1}$. Since $a\overline{a}+b\overline{b}\gamma^{2}=1$, we obtain that $a\overline{a}\in 1 + \mathfrak{p}_{F}^{2(r+1)}$, as the norm map $N_{E/F}: 1+\mathfrak{p}_{E}^{2(r+1)}\rightarrow 1+\mathfrak{p}_{F}^{2(r+1)}$ is surjective, there exists $c\in 1+\mathfrak{p}_{E}^{2(r+1)}$ such that $c\overline{c}=a\overline{a}$. Hence we may write 
	\begin{align*}
		t^{g_{d}^{-1}}=\begin{pmatrix}a+b\gamma& 0\\ 0& a-b\gamma\end{pmatrix}=\begin{pmatrix}ac^{-1}&0\\ 0&ac^{-1}\end{pmatrix}\begin{pmatrix}c(1+a^{-1}b\gamma)&0\\ 0& c(1-a^{-1}b\gamma)\end{pmatrix}\in ZT_{r+1}
	\end{align*}
	Since $\chi$ has depth $r$, we conclude that $\zeta \mid_{T(Y_{\chi})}=\theta\mid_{Z}$ as required. Finally,  by Lemma~\ref{remark3.3.4}, we have
	$$\mathcal{S}_{d}(Y_{\chi}, \zeta_{\chi})=\mathcal{S}_{d}(X_{\epsilon^{-1}\varpi^{-d}}, \theta)\cong \mathcal{S}_{d}(X_{\varpi^{-d}}, \theta).$$
\end{proof}	

Recall that we denoted by $\delta$ the non-trivial quadratic character of $E^{1}$. 
Since \( T \cong E^{\times} \) and \( Z \cong E^{1} \), we can extend $\delta$ to a character 
\(\chiE_{\delta}\) of \( E^{\times} \), and we denote by \(\chi_{\delta}\) the corresponding character of \( T \). 
Let \( \tau_{0} = \mathrm{Ind}_{B}^{G}\mathbbm{1} \) and 
\( \tau_{1} = \mathrm{Ind}_{B}^{G}\chi_{\delta} \) be the corresponding principal series representations. 
Then the central characters of \( \tau_{0} \) and \( \tau_{1} \) are \( \mathbbm{1} \) and \( \delta \), respectively. 

The following corollary, which follows immediately from 
Theorem~\ref{keyidentificationp} and 
Lemma~\ref{centralcharacterofprincipalserieshasdepth-zero}, 
answers the Question 1.2 stated in~\cite{guy2024representationsglndnearidentity} in the context of principal series representations of $G$.
\begin{corollary}
	For any principal series representation \( \pi_{\chi} \) of \( G \), 
	the higher-depth components in the decomposition of 
	\( \pi_{\chi}\!\mid_{\mathcal{K}} \) coincide, up to a twist by a character of \( G \), 
	with those appearing in the decompositions of 
	\( \tau_{0}\!\mid_{\mathcal{K}} \) or \( \tau_{1}\!\mid_{\mathcal{K}} \).
\end{corollary}

\subsection{Representations associated with nilpotent orbits}\label{representationsassociatedwithnilpotentorbits}

Recall from section~\ref{chapter3} that there are three nilpotent $G$-orbits in $\mathfrak{g}$, and they are parametrized by $\left\{ X_{\delta} \;\middle\vert\; \delta \in \{0, 1, \varpi\} \right\},$ where
$X_{\delta} = \begin{psmallmatrix} 0 & \delta \sqrt{\epsilon} \\ 0 & 0 \end{psmallmatrix}.
$

For $\delta \in \{0, 1, \varpi\}$, we recall that the nilpotent orbits were denoted by $\mathcal{N}_{\delta}$, and that each $G$-orbit $\mathcal{N}_{\delta}$ decomposes into the following $\mathcal{K}$-orbits
\[
\mathcal{N}_{0} = \mathcal{K} \cdot X_{0}, \hspace{3em}
\mathcal{N}_{1} = \bigsqcup_{m \in 2\mathbb{Z}} \mathcal{K} \cdot X_{\varpi^{-m}}, \hspace{3em}
\mathcal{N}_{\varpi} = \bigsqcup_{n \in 2\mathbb{Z}+1} \mathcal{K} \cdot X_{\varpi^{-n}}.
\]
For each depth-zero character $\theta$ of the center of $G$, we define  highly reducible representations 
\[
\begin{aligned}
\tau_{\mathcal{N}_{0}}(\theta) &= \mathbbm{1}, \hspace{2em}	\tau_{\mathcal{N}_{1}}(\theta) &= \bigoplus_{d \in 2\mathbb{Z}_{> 0}} \mathcal{S}_{d}(X_{\varpi^{-d}}, \theta), \hspace{2em}
	\tau_{\mathcal{N}_{\varpi}}(\theta) = \bigoplus_{d \in 2\mathbb{Z}_{\geq 0} + 1 } \mathcal{S}_{d}(X_{\varpi^{-d}}, \theta).
\end{aligned}
\]

To prepare for the decomposition of $\pi_{\rho}|_{\mathcal{K}_{2r+}}$, we compute the dimensions of the $\mathcal{K}_{2r+}$-fixed vectors in the representations defined above.
\begin{lemma}\label{dimensionofrepresentationsassociatedtonilpotentorbitsofG}
	Let $\theta$ be a depth-zero character of the center of $G$. Then
	\[
	\begin{aligned}
		\dim\!\left(\tau_{\mathcal{N}_{1}}(\theta)^{\mathcal{K}_{2r+}}\right) &= q \left(q^{2r} - 1\right),\;\; \text{and}\;\;
		\dim\!\left(\tau_{\mathcal{N}_{\varpi}}(\theta)^{\mathcal{K}_{2r+}}\right) = q^{2r} - 1.
	\end{aligned}
	\]
\end{lemma}
\begin{proof}
	For each $d \in \mathbb{Z}_{>0}$, the representation $\mathcal{S}_{d}(X_{\varpi^{-d}}, \theta)$ is irreducible of depth $d$. Hence, the space of $\mathcal{K}_{2r+}$-fixed vectors in $\tau_{\mathcal{N}_{\delta}}(\theta)$ is obtained by adding the dimensions of those $\mathcal{S}_{d}(X_{\varpi^{-d}}, \theta)$ with $d \leq 2r$. Moreover, for each $d \in \mathbb{Z}_{>0}$ the dimension of $\mathcal{S}_{d}(X_{\varpi^{-d}}, \theta)$ is given by $(q^{2}-1) q^{d-1}$. Therefore,
	\[
	\begin{aligned}
		\dim\!\left(\tau_{\mathcal{N}_{1}}(\theta)^{\mathcal{K}_{2r+}}\right)
		&= \sum_{\substack{d \in 2\mathbb{Z}_{>0} \\ d \leq 2r}} (q^{2}-1) q^{d-1}
		= (q^{2}-1) \sum_{t=1}^{r} q^{2t-1}
		= q \left(q^{2r} - 1\right), \\
		\dim\!\left(\tau_{\mathcal{N}_{\varpi}}(\theta)^{\mathcal{K}_{2r+}}\right)
		&= \sum_{\substack{d \in 2\mathbb{Z}_{\geq 0}+1 \\ d \leq 2r}} (q^{2}-1) q^{d-1}
		= (q^{2}-1) \sum_{t=0}^{r-1} q^{2t-1}
		= q^{2r} - 1.
	\end{aligned}
	\]
\end{proof}

\subsection{Restriction to $\mathcal{K}_{2r+}$}
	In this section, we prove a representation theoretic version of the local character expansion in analogy with~\cite[Theorem~1.1]{Mon2024}.
	
	\begin{theorem}\label{principalseriesrestrictedtoK2r}
			Let $\chi$ be a character of \( T \) of minimal depth \( r \in \mathbb{Z}_{\geq 0} \), and let \( \pi_{\chi} \) be the associated principal series representation. Then
			\[
			\mathrm{Res}_{\mathcal{K}_{2r+1}}\pi_{\pi_{\chi}}
			\;\cong\;
			(q+1) \mathbbm{1}
			\;+\;
			\mathrm{Res}_{\mathcal{K}_{2r+1}} \tau_{\mathcal{N}_{1}}(\theta)
			\;+\;
			\mathrm{Res}_{\mathcal{K}_{2r+1}} \tau_{\mathcal{N}_{\varpi}}(\theta).
			\]
		\end{theorem}
		\begin{proof} We first assume that $\chi$ is a character of $T$ of minimal depth $r>0$.  By Theorem~\ref{thm3} the restriction of $\pi_{\chi}$ to $\mathcal{K}$ has the following decomposition
			$$\mathrm{Res}_{\mathcal{K}}\pi_{\chi}\cong V_{\chi}^{\mathcal{K}_{r+1}}\oplus\bigoplus_{d\geq r+1}\mathcal{S}_{d}(Y_{\chi}, \zeta_{\chi}).$$ The restriction of $\pi_{\chi}$  to the subgroup $\mathcal{K}_{2r+1}$ will act trivially on the components $\mathcal{S}_{d}(Y_{\chi}, \zeta_{\chi})$ that have depth less than or equal to $2r$. Therefore, applying Theorem~\ref{keyidentificationp} we obtain 
			$$\mathrm{Res}_{\mathcal{K}_{2r+1}}\pi_{\chi}\cong V_{\chi}^{\mathcal{K}_{2r+1}}\oplus\bigoplus_{d>2r }\mathrm{Res}_{\mathcal{K}_{2r+1}}\mathcal{S}_{d}(X_{\varpi^{-d}}, \theta).$$
			Hence, in the Grothendieck group of representations we have the desired decomposition
			\[
			\mathrm{Res}_{\mathcal{K}_{2r+1}}\pi_{\pi_{\chi}}
			\;\cong\;
			n(\pi_{\chi}) \mathbbm{1}
			\;+\;
			\mathrm{Res}_{\mathcal{K}_{2r+1}} \tau_{\mathcal{N}_{1}}(\theta)
			\;+\;
			\mathrm{Res}_{\mathcal{K}_{2r+1}} \tau_{\mathcal{N}_{\varpi}}(\theta),
			\]
			where $n(\pi_{\chi})=\mathrm{dim}(V_{\chi}^{\mathcal{K}_{2r+1}})- \dim\!\left(\tau_{\mathcal{N}_{1}}(\theta)^{\mathcal{K}_{2r+1}}\right)
			- \dim\!\left(\tau_{\mathcal{N}_{\varpi}}(\theta)^{\mathcal{K}_{2r+1}}\right)$.
			By Lemma~\ref{dimensionofVchiKn}, we have $\mathrm{dim}\left(V_{\chi}^{\mathcal{K}_{2r+1}}\right)=(q+1)q^{2r}$, and by Lemma~\ref{dimensionofrepresentationsassociatedtonilpotentorbitsofG}, we have 
			$\dim\!\left(\tau_{\mathcal{N}_{1}}(\theta)^{\mathcal{K}_{2r+1}}\right) = q \left(q^{2r} - 1\right)$, and $
			\dim\!\left(\tau_{\mathcal{N}_{\varpi}}(\theta)^{\mathcal{K}_{2r+1}}\right) = q^{2r} - 1$. Therefore,
			$$n(\pi_{\chi})=(q+1)q^{2r}-q(q^{2r}-1)-(q^{2r}-1)=q+1.$$
			Now let us assume that $\chi$ has minimal depth $0$. Then by Corollary~\ref{thm3} and Theorem~\ref{keyidentificationp} we have 
			$$\mathrm{Res}_{\mathcal{K}_1}\pi_{\chi}\cong V_{\chi}^{\mathcal{K}_1}\oplus\bigoplus_{d\geq 1}\mathrm{Res}_{\mathcal{K}_{1}}\mathcal{S}_{d}(X_{\varpi^{-d}}, \theta)$$
			where $V_{\chi}^{\mathcal{K}_{1}}=\mathbbm{1}_{q}\oplus \mathrm{St}_{q}$ if $\chi=\mathbbm{1}$, and an irreducible representation of degree $(q+1)$ otherwise. 
			Therefore we have the following decomposition of $\pi_{\chi}$ when restricted to $\mathcal{K}_{1}$
			$$\mathrm{Res}_{\mathcal{K}_1}\pi_{\chi}\cong \mathrm{n}(\pi_{\chi})\mathbbm{1}\oplus 	\mathrm{Res}_{\mathcal{K}_{1}}\tau_{\mathcal{N}_{1}}(\theta)\oplus 	\mathrm{Res}_{\mathcal{K}_{1}}\tau_{\mathcal{N}_{\varpi}}(\theta)$$
where $n(\pi_{\chi})=\mathrm{dim}(V_{\chi}^{\mathcal{K}_{1}})=(q+1)$ by Lemma~\ref{dimensionofVchiKn}.
			\end{proof}

	\appendix	
	\section{The magic of Hensel's lemma}\label{section1ofappendix}		
	In this appendix, we provide a detailed proof that is essentially an application of Hensel’s lemma, which was postponed from Proposition~\ref{centralizerofXuv1} in section~\ref{chapter3}.

	\begin{lemma}
		Let $s>0$.  Suppose $k=\begin{pmatrix}a&b\\c&d\end{pmatrix}\in \mathcal{K}$ satisfies $a\equiv d \mod \mathfrak{p}_{E}^s$ and $c\equiv u^{-1}vb \mod \mathfrak{p}_{E}^s$, where $u^{-1}v\in \mathcal{O}_{F}$ and $\val(u)=0$ and $\val(v)>0$. Then there exists a matrix $k' = \begin{pmatrix}a' & b'\\ b'u^{-1}v & a'\end{pmatrix}\in \mathcal{K}$ such that $a' \equiv a \mod \mathfrak{p}_{E}^s$ and $b' \equiv b\mod \mathfrak{p}_{E}^s$; in particular, $(k')^{-1}k\in \mathcal{K}_{s}$.
	\end{lemma}
	
	\begin{proof}
		We wish to find $x,y\in \mathcal{O}_{E}$ such that setting $a'=a+x\varpi^s$ and $b'=b+y\varpi^s$ yields a matrix $k'\in \mathcal{K}$.  That is, we require the following defining identities of $U(1,1)$ to hold:
		\begin{align*}
			\overline{(a+x\varpi^s)}(a+x\varpi^s) &+ u^{-1}v\overline{(b+y\varpi^s)}(b+y\varpi^s) = 1\\
			\overline{(a+x\varpi^s)}u^{-1}v(b+y\varpi^s) &\in \sqrt{\varepsilon}F\\
			\overline{(b+y\varpi^s)}(a+x\varpi^s) &\in \sqrt{\varepsilon}F
		\end{align*}
		Since $u^{-1}v\in F$, the second equation follows from the third.  We produce the solution $(x,y)$ by induction. 
		
		We begin with depth $t\leq s$.  Then choosing $x=y=0$ is a solution modulo $\mathfrak{p}_{E}^t$.  Suppose we have a pair $(x,y)$ that solves this system modulo $\mathfrak{p}_{E}^t$. Then setting $a'=a+x\varpi^s$, $b'=b+y\varpi^s$ we know that there exists $\alpha, \gamma \in \mathcal{O}_{F}$ and  $\beta\in \mathcal{O}_{E}$,  such that 
		\begin{equation}\label{11}
			\overline{a'}a' + u^{-1}v\overline{b'}b' = 1 + \alpha\varpi^{t}
		\end{equation}
		and 
		\begin{equation}\label{22}
			\overline{b'}a' = \sqrt{\varepsilon}\gamma + \beta\varpi^{t}
		\end{equation}
		
		We claim that we can find $x',y'\in \mathcal{O}_{E}$ such that $(a'',b'')=(a'+x'\varpi^{t}, b'+y'\varpi^{t})$ satisfies
		\begin{equation}\label{111}
			\overline{a''}a'' + u^{-1}v\overline{b''}b'' \in 1+\mathfrak{p}_{E}^{t+1}
		\end{equation}
		and \begin{equation}\label{222}
			\overline{b''}a'' \in \sqrt{\varepsilon}\gamma +\mathfrak{p}_{E}^{t+1}.
		\end{equation}
		
		Write $a'=a'_0+a'_1\sqrt{\varepsilon}$, $b' = b'_0+b'_1\sqrt{\varepsilon}$ and $\beta=\beta_0+\beta_1\sqrt{\varepsilon}$; our unknowns are  $x'=x_0'+x_1'\sqrt{\varepsilon}$ and $y'=y_0'+y_1'\sqrt{\varepsilon}$.  We will show that $x_0', x_1', y_0', y_1'$ are a solution to a nondegenerate linear system over the residue field $\mathfrak{f}$ of $F$.
		
		Expanding the left hand side of \eqref{111} and subtracting the left hand side of \eqref{11} reveals that we need to solve only $a'\overline{x'}+\overline{a'}x' \equiv -\alpha \mod \mathfrak{p}_{F}$.  In terms of our coefficients, we need to satisfy
		\begin{align*}
			2(a'_0x_0' - \varepsilon a_1'x_1') &\equiv -\alpha \mod \mathfrak{p}_{F}.
		\end{align*}
		Similarly, expanding the left hand side of \eqref{222} and subtracting the left hand side of \eqref{22} reveals that we only require $\overline{b'}x'+a'\overline{y'} \equiv -\beta \mod \mathfrak{p}_{E}$.  In terms of our coefficients, this is the system
		\begin{align*}
			b_0'x_0' - b_1'\varepsilon x'_1 +a'_0y_0'-a_1'\varepsilon y_1' &\equiv -\beta_0 \mod \mathfrak{p}_{F}\\
			-b_1' x_0' + b_0' x_1' +a_1' y_0' - a_0' y_1' &\equiv -\beta_1 \mod \mathfrak{p}_{F}.
		\end{align*}
		Since our original matrix $k$ has determinant in $\mathcal{O}_{E}^\times$, and $a'\equiv a\mod \mathfrak{p}_{E}$, $b'\equiv b\mod \mathfrak{p}_{E}$, we infer that $(a'_0)^2-\varepsilon (a'_1)^2 \in \mathcal{O}_{E}^\times$, whence the linear system composed of the three preceding displayed equations is consistent over the field $\mathcal{O}_{F}/\mathfrak{p}_{F}$, and thus has a  solution.  By the principle of mathematical induction, this process gives as a limit an element $k'\in \mathcal{K}$ satisfying the lemma.
	\end{proof}

\providecommand{\bysame}{\leavevmode\hbox to3em{\hrulefill}\thinspace}
\providecommand{\MR}{\relax\ifhmode\unskip\space\fi MR }
\providecommand{\MRhref}[2]{%
	\href{http://www.ams.org/mathscinet-getitem?mr=#1}{#2}
}
\providecommand{\href}[2]{#2}


\begin{thebibliography}{AOPS10}
	
	\bibitem[AOPS10]{Aubonnprasta2010}
	Anne-Marie Aubert, Uri Onn, Amritanshu Prasad, and Alexander Stasinski,
	\emph{On cuspidal representations of general linear groups over discrete
		valuation rings}, Israel J. Math. \textbf{175} (2010), 391--420. \MR{2607551}
	
	\bibitem[BH06]{Bus06}
	Colin~J. Bushnell and Guy Henniart, \emph{The local {L}anglands conjecture for
		{$\rm GL(2)$}}, Grundlehren der mathematischen Wissenschaften [Fundamental
	Principles of Mathematical Sciences], vol. 335, Springer-Verlag, Berlin,
	2006. \MR{2234120}
	
	\bibitem[Car79]{Car79}
	P.~Cartier, \emph{Representations of {$p$}-adic groups: a survey}, Automorphic
	forms, representations and {$L$}-functions ({P}roc. {S}ympos. {P}ure {M}ath.,
	{O}regon {S}tate {U}niv., {C}orvallis, {O}re., 1977), {P}art 1, Proc. Sympos.
	Pure Math., vol. XXXIII, Amer. Math. Soc., Providence, RI, 1979,
	pp.~111--155. \MR{546593}
	
	\bibitem[Cas73]{Cas71}
	William Casselman, \emph{The restriction of a representation of {${\rm
				GL}\sb{2}(k)$} to {${\rm GL}\sb{2}({\mathfrak o})$}}, Math. Ann. \textbf{206}
	(1973), 311--318. \MR{338274}
	
	\bibitem[CF67]{Cas1967}
	J.~W.~S. Cassels and A.~Fr\"ohlich (eds.), \emph{Algebraic number theory},
	Academic Press, London; Thompson Book Co., Inc., Washington, DC, 1967.
	\MR{215665}
	
	\bibitem[CMO24]{Onn2024}
	Tyrone Crisp, Ehud Meir, and Uri Onn, \emph{An inductive approach to
		representations of general linear groups over compact discrete valuation
		rings}, Adv. Math. \textbf{440} (2024), Paper No. 109516, 58. \MR{4704476}
	
	\bibitem[CN09]{CamMon2009}
	Peter~S. Campbell and Monica Nevins, \emph{Branching rules for unramified
		principal series representations of {${\rm GL}(3)$} over a {$p$}-adic field},
	J. Algebra \textbf{321} (2009), no.~9, 2422--2444. \MR{2504482}
	
	\bibitem[CN10]{CamMon2010}
	\bysame, \emph{Branching rules for ramified principal series representations of
		{GL}(3) over a {$p$}-adic field}, Canad. J. Math. \textbf{62} (2010), no.~1,
	34--51. \MR{2597022}
	
	\bibitem[GF16]{Frez2016}
	Luis Guti\'errez~Frez, \emph{Construction of primitive representations of
		{${\rm U}(1,1)(\mathcal{O})$}}, J. Lie Theory \textbf{26} (2016), no.~3,
	691--716. \MR{3447945}
	
	\bibitem[GMF24]{guy2024representationsglndnearidentity}
	Henniart Guy and Vignéras Marie-France, \emph{Representations of
		$\mathrm{GL}_n({D})$ near the identity}, 2024.
	
	\bibitem[Han87]{Kri87}
	Kristina Hansen, \emph{Restriction to {${\rm GL}_2({\mathcal O})$} of
		supercuspidal representations of {${\rm GL}_2(F)$}}, Pacific J. Math.
	\textbf{130} (1987), no.~2, 327--349. \MR{914105}
	
	\bibitem[Hil94]{Hill1994}
	Gregory Hill, \emph{On the nilpotent representations of {${\rm
				GL}_n(\mathcal{O})$}}, Manuscripta Math. \textbf{82} (1994), no.~3-4,
	293--311. \MR{1265002}
	
	\bibitem[Hil95]{Hill21995}
	\bysame, \emph{Semisimple and cuspidal characters of {${\rm
				GL}_n(\mathcal{O})$}}, Comm. Algebra \textbf{23} (1995), no.~1, 7--25.
	\MR{1311772}
	
	\bibitem[Hum12]{humphreys2012}
	James~E Humphreys, \emph{Introduction to lie algebras and representation
		theory}, vol.~9, Springer Science \& Business Media, 2012.
	
	\bibitem[HV24]{henniart2024representationssl2f}
	Guy Henniart and Marie-France Vignéras, \emph{Representations of
		$\mathrm{SL}_2({F})$}, 2024.
	
	\bibitem[Klo46a]{Kl1946}
	H.~D. Kloosterman, \emph{The behaviour of general theta functions under the
		modular group and the characters of binary modular congruence groups. {I}},
	Ann. of Math. (2) \textbf{47} (1946), 317--375. \MR{21032}
	
	\bibitem[Klo46b]{Kl21946}
	\bysame, \emph{The behaviour of general theta functions under the modular group
		and the characters of binary modular congruence groups. {II}}, Ann. of Math.
	(2) \textbf{47} (1946), 376--447. \MR{21033}
	
	\bibitem[KN25]{ZanMon2025}
	Zander Karaganis and Monica Nevins, \emph{Branching rules for irreducible
		depth-zero supercuspidal representations of $\mathrm{SL}(2,{F})$, when ${F}$
		has residual characteristic $2$}, 2025.
	
	\bibitem[MT11]{makpie2011}
	Khemais Maktouf and Pierre Torasso, \emph{Restriction de la repr\'esentation de
		{W}eil \`a un sous-groupe compact maximal ou \`a un tore maximal elliptique},
	2011.
	
	\bibitem[Nag81]{Nag1981}
	S.~V. Nagornyi, \emph{Complex representations of the group $\mathrm{GL}(2,
		\mathbb{Z}/p^n\mathbb{Z})$}, Journal of Soviet Mathematics \textbf{17}
	(1981), 1777--1783.
	
	\bibitem[Nev05]{Mon2005}
	Monica Nevins, \emph{Branching rules for principal series representations of
		{${\rm SL}(2)$} over a {$p$}-adic field}, Canad. J. Math. \textbf{57} (2005),
	no.~3, 648--672. \MR{2134405}
	
	\bibitem[Nev13]{MN13}
	\bysame, \emph{Branching rules for supercuspidal representations of {${\rm
				SL}_2(k)$}, for {$k$} a {$p$}-adic field}, J. Algebra \textbf{377} (2013),
	204--231. \MR{3008903}
	
	\bibitem[Nev24]{Mon2024}
	\bysame, \emph{The local character expansion as branching rules: nilpotent
		cones and the case of {${\rm SL}(2)$}}, Pacific J. Math. \textbf{329} (2024),
	no.~2, 259--301. \MR{4767894}
	
	\bibitem[Nob76]{Nob1976}
	Alexandre Nobs, \emph{La s\'erie d\'eploy\'ee et la s\'erie non ramifi\'ee de
		$\mathrm {GL}_{2}(\mathcal{O})$}, C. R. Acad. Sci. Paris S\'er. A-B
	\textbf{283} (1976), no.~6, Aii, A297--A300. \MR{447489}
	
	\bibitem[NW76]{NobWolf1976}
	Alexandre Nobs and J\"urgen Wolfart, \emph{Die irreduziblen {D}arstellungen der
		{G}ruppen {$SL\sb{2}(Z\sb{p})$}, insbesondere {$SL\sb{2}(Z\sb{p})$}. {II}},
	Comment. Math. Helv. \textbf{51} (1976), no.~4, 491--526. \MR{444788}
	
	\bibitem[OS14]{PooOnn2014}
	Uri Onn and Pooja Singla, \emph{On the unramified principal series of
		{$\mathrm{GL}(3)$} over non-{A}rchimedean local fields}, J. Algebra
	\textbf{397} (2014), 1--17. \MR{3119211}
	
	\bibitem[Pra98]{Dip98}
	Dipendra Prasad, \emph{A brief survey on the theta correspondence}, Number
	theory ({T}iruchirapalli, 1996), Contemp. Math., vol. 210, Amer. Math. Soc.,
	Providence, RI, 1998, pp.~171--193. \MR{1478492}
	
	\bibitem[Sha04]{Sha2004}
	Joseph~A. Shalika, \emph{Representation of the two by two unimodular group over
		local fields}, Contributions to automorphic forms, geometry, and number
	theory, Johns Hopkins Univ. Press, Baltimore, MD, 2004, pp.~1--38.
	\MR{2058601}
	
	\bibitem[Sil70]{Sil1970}
	Allan~J. Silberger, \emph{{${\rm PGL}\sb{2}$} over the {$p$}-adics: its
		representations, spherical functions, and {F}ourier analysis}, Lecture Notes
	in Mathematics, vol. Vol. 166, Springer-Verlag, Berlin-New York, 1970.
	\MR{285673}
	
	\bibitem[Sil77]{Sil1977}
	\bysame, \emph{Irreducible representations of a maximal compact subgroup of
		$\mathrm{PGL}_{2}$ over the {$p$}-adics}, Math. Ann. \textbf{229} (1977),
	no.~1, 1--12. \MR{463366}
	
	\bibitem[Sin10]{Singla2010}
	Pooja Singla, \emph{On representations of general linear groups over principal
		ideal local rings of length two}, J. Algebra \textbf{324} (2010), no.~9,
	2543--2563. \MR{2684153}
	
	\bibitem[Sta09]{Sta2009}
	Alexander Stasinski, \emph{The smooth representations of {${\mathrm
				{GL}}_2(\mathfrak{ o})$}}, Comm. Algebra \textbf{37} (2009), no.~12,
	4416--4430. \MR{2588859}
	
	\bibitem[Tan66]{Tan1966}
	Shun'ichi Tanaka, \emph{On irreducible unitary representations of some special
		linear groups of the second order. {I}, {II}}, Osaka Math. J. \textbf{3}
	(1966), 217--227; 229--242. \MR{223493}
	
	\bibitem[Tan67]{Tan1967}
	\bysame, \emph{Irreducible representations of the binary modular congruence
		groups {${\rm mod}\,p\sp{\lambda }$}}, J. Math. Kyoto Univ. \textbf{7}
	(1967), 123--132. \MR{229737}
	
	\bibitem[Tiw25]{ET25}
	Ekta Tiwari, \emph{Branching rules for irreducible supercuspidal
		representations of unramified $\mathrm{U}(1,1)$}, Preprint, 2025.
	
\end{thebibliography}
\end{document}